\documentclass[onecolumn, journal]{IEEEtran}  % Comment this line out
\IEEEoverridecommandlockouts                              % This command is only
\usepackage[utf8]{inputenc}
\usepackage{amsmath,bm}
\allowdisplaybreaks
\usepackage{amsfonts}
\usepackage{amssymb}
\usepackage{graphicx}
\usepackage{amsfonts}
\usepackage{hyperref}
\usepackage{caption}
\usepackage{subcaption}
\usepackage{algorithm}
\usepackage[noend]{algpseudocode}
\usepackage{mathtools}
\usepackage{xcolor}
\usepackage{cite}
\usepackage[normalem]{ulem}
\usepackage{textcomp}

\newtheorem{Theorem}{Theorem}
\newtheorem{Proposition}{Proposition}
\newtheorem{Lemma}{Lemma}
\newtheorem{Problem}{Problem}
\newtheorem{Remark}{Remark}
\newtheorem{Assumption}{Assumption}

\usepackage{tikz}
	\usetikzlibrary{shadows,shapes,arrows}
	\tikzstyle{frame} = [draw, -latex]
	\tikzstyle{line} = [draw]
	\tikzstyle{line2} = [draw, dashdotted]
	\tikzstyle{line3} = [draw, dashed]
	\tikzstyle{line3UD} = [draw, dashed]
	\tikzstyle{place} = [circle, draw=black, fill=white, thick, inner sep=2pt, minimum size=1mm]
	\tikzstyle{place2} = [circle, draw=black, fill=black, thick, inner sep=2pt, minimum size=1mm]
	\tikzstyle{placeRed} = [circle, draw=red, fill=red, thick, inner sep=2pt, minimum size=1mm]
	\tikzstyle{vertex} = [circle, draw=black, fill=black, thick, inner sep=2pt, minimum size=1mm]
\usepackage{tkz-euclide}
\usepackage{multirow}

\makeatletter
\def\algbackskip{\hskip-\ALG@thistlm}
\makeatother

\usepackage[switch,columnwise]{lineno}
\usepackage{setspace}

\title{Distributed solution methods for MPC based energy management method of interconnected microgrids: dual ascent vs ADMM}
\author{Viet Hoang Pham and Hyo-Sung Ahn 
}

\begin{document}
%%%%%%%%%%%%%%%%%%%%%%%%%%%%%%%%%%%%%%%%%%%%%%%%%%%%%%%%%%%%%%%%%%%%%%%%%%%%%%%%%%%%%%%%%%
%\pagewiselinenumbers
\maketitle 
\thispagestyle{empty}
\pagestyle{empty}
%%%%%%%%%%%%%%%%%%%%%%%%%%%%%%%%%%%%%%%%%%%%%%%%%%%%%%%%%%%%%%%%%%%%%%%%%%%%%%%%%%%%%%%%%%
%%%%%%%%%%%%%%%%%%%%%%%%%%%%%%%%%%%%%%%%%%%%%%%%%%%%%%%%%%%%%%%%%%%%%%%%%%%%%%%%%%%%%%%%%%
\begin{abstract}
%%%%%%%%%%%%%%%%%%%%%%%%%%%%%%%%%%%%%%%%%%%%%%%%%%%%%%%%%%%%%%%%%%%%%%%%%%%%%%%%%%%%%%%%%%
This paper considers an optimal energy management problem for a network of interconnected microgrids. A model predictive control (MPC) approach is used to avoid capacity constraint violation and to cope with uncertainties of forecasted power demands.
By employing a dual ascent method and a proximal alternative direction multiplier method (ADMM), respectively, two distributed methods are designed to allow every agent using only local information to determine its own optimal control decisions.
The effectiveness of the proposed method is verified via numerical simulations.
%%%%%%%%%%%%%%%%%%%%%%%%%%%%%%%%%%%%%%%%%%%%%%%%%%%%%%%%%%%%%%%%%%%%%%%%%%%%%%%%%%%%%%%%%%
\end{abstract}
%%%%%%%%%%%%%%%%%%%%%%%%%%%%%%%%%%%%%%%%%%%%%%%%%%%%%%%%%%%%%%%%%%%%%%%%%%%%%%%%%%%%%%%%%%
%%%%%%%%%%%%%%%%%%%%%%%%%%%%%%%%%%%%%%%%%%%%%%%%%%%%%%%%%%%%%%%%%%%%%%%%%%%%%%%%%%%%%%%%%%
\section{Introduction}
%%%%%%%%%%%%%%%%%%%%%%%%%%%%%%%%%%%%%%%%%%%%%%%%%%%%%%%%%%%%%%%%%%%%%%%%%%%%%%%%%%%%%%%%%%
With the rising power demands and the preferred use of renewable and clean energy resources, a power network consisting of multiple power nodes becomes a desired structure in power systems \cite{MLTuballa2016, XinghuoYu2016}. Each node may consist of power generators, power consumers, and energy storage units. Many strategies to optimize power distribution and transmission have been proposed and inquired for this power network \cite{GuillermoIrisarri1998, XSHan2001, XSHan2007, ShipingYang2013, AshishCherukuri2015, YinliangXu2015, RuiWang2018, HyoSungAhn2018}.
The economic dispatch problem (EDP) is a crucial topic in managing power networks. The main objective is to find the optimal power distribution among generators, which minimizes the total generation cost while satisfying limitation constraints and balancing the power demand and supply \cite{XSHan2001, ShipingYang2013, AshishCherukuri2015, YinliangXu2015, RuiWang2018}. Besides that, constraints on power transmission lines are considered to guarantee feasible flow among power nodes \cite{GuillermoIrisarri1998, XSHan2007}.
In \cite{HyoSungAhn2018, VietHoangPham_ICCAIS2022}, the power transmission cost is considered in coordination with the power distribution problem. In addition, to enhance the flexibility of the power network, some nodes have energy storage units \cite{DiWu2017, TaoYang2019}. Due to the geographical spread of the power network, the distributed control setup has been more desired \cite{ShipingYang2013, AshishCherukuri2015, YinliangXu2015, RuiWang2018, HyoSungAhn2018, DiWu2017, TaoYang2019}.

Short-term power load forecasting \cite{SeyedehNarjesFallah2019}, which focuses on hours-ahead prediction, is essential in managing a power network \cite{JiayuHan2021}. However, power demands can not be predicted precisely due to the fluctuation of historical data and unpredictable events. Model predictive control (MPC) \cite{SJoeQin2003} approach has been successfully applied in various areas such as the process control industry, traffic control, and recently in power networks \cite{XSHan2001, XSHan2007, YinliangXu2015}. The formulation of a finite horizon optimal control problem using updated states and prediction information at each sampling instant and the repetition over time can help the power network operate effectively even the uncertainties in power demand forecasting.

In this paper, we formulate the energy management problem of a power network in the MPC approach to find the optimal plan of power generation, transmission and energy storage operation. Due to the uncertainties in forecasting, the power demands forecasted can not be predicted precisely. Instead, we assume that their expectations and variations can be found with high confidence. Assuming that each power node has one local controller and consider this controller as an agent in a multiagent system. The main target of this paper is to design two distributed control methods for each agent using only local information to determine its own optimal control decisions. The first distributed method is based on a dual ascent method. The second one is based on a proximal ADMM. 

The rest of this paper is organized as follows. Section II introduces the power network and formulates the coordination problem of power distribution and transmission in the MPC approach. Then the distributed control problem is stated in Section III. We design two distributed methods for each local agent to find its optimal control decisions in Section IV and Section V. In Section VI, the effectiveness of these proposed methods is verified by numerical simulations for a modified IEEE 30-bus test network. Finally, Section VI provides the conclusion.
%%%%%%%%%%%%%%%%%%%%%%%%%%%%%%%%%%%%%%%%%%%%%%%%%%%%%%%%%%%%%%%%%%%%%%%%%%%%%%%%%%%%%%%%%%
%%%%%%%%%%%%%%%%%%%%%%%%%%%%%%%%%%%%%%%%%%%%%%%%%%%%%%%%%%%%%%%%%%%%%%%%%%%%%%%%%%%%%%%%%%
\subsection{Notations}
%%%%%%%%%%%%%%%%%%%%%%%%%%%%%%%%%%%%%%%%%%%%%%%%%%%%%%%%%%%%%%%%%%%%%%%%%%%%%%%%%%%%%%%%%%
We use $\mathbb{R}$, $\mathbb{R}^n$, and $\mathbb{R}^{m \times n}$ to denote the set of real numbers, the set of $n$-dimension real vectors, and the set of $m \times n$ matrices, respectively. The set of all $n$-dimensional vector whose elements are all non-negative real numbers is denoted by $\mathbb{R}_{+}^n$.
Let $\textbf{1}_{n}$ and $\textbf{0}_{n}$ be the vectors in $\mathbb{R}^{n}$ whose all elements are $1$ and $0$, respectively. We use $\textbf{I}_{n}$ to denote the identity matrix in $\mathbb{R}^{n \times n}$. When the dimension $n$ is clear, we can ignore it and write the vectors and matrix as $\textbf{1}, \textbf{0}$ and $\textbf{I}$.
Given a matrix $\textbf{H}$ and a vector $\textbf{h}$, $\textbf{H}^T$ denotes the transpose of $\textbf{H}$, $[\textbf{H}]_{ij}$ is the entry in row $i$ and column $j$ of the matrix $\textbf{H}$, and $[\textbf{h}]_i$ denotes the $i^{th}$-element of $\textbf{h}$. Let $\rho_{min}(\textbf{H})$ and $\rho_{max}(\textbf{H})$ present the smallest and largest nonzero singular values of the matrix $\textbf{H}$, respectively. We use $\otimes$ to denote Kronecker product.

For a given set $\mathcal{H}$, $|\mathcal{H}|$ represents the cardinality of this set. Let $\mathcal{H} = \{\textbf{h}_1, \textbf{h}_2, \dots, \textbf{h}_n\}$ be a set of $n$ vectors, we use $col \mathcal{H}$ to denote the column vector
\[\textrm{col}\mathcal{H} = \textrm{col}\{\textbf{h}_1, \textbf{h}_2, \dots, \textbf{h}_n\} = [\textbf{h}_1^T, \textbf{h}_2^T, \dots, \textbf{h}_n^T]^T.\]
When the set $\mathcal{H}$ has a finite number of matrices, i.e., $\mathcal{H} = \{\textbf{H}_1, \textbf{H}_2, \dots, \textbf{H}_n\}$, we use $\textrm{blkcol}\mathcal{H}$ to denote the following matrix
\[\textrm{blkcol}\mathcal{H} = \left[\begin{matrix}\textbf{H}_1\\ \textbf{H}_2\\ \vdots\\ \textbf{H}_N\end{matrix}\right]\]
For two vectors $\textbf{h} = \textrm{col}\bigl\{h_r: 1 \le r \le n\bigr\}$ and $\textbf{g} = \textrm{col}\bigl\{g_r: 1 \le r \le n\bigr\}$, we have $\textbf{h} \ge \textbf{g}$ if $h_r \ge g_r$ for all $r \in [1,n]$.
The Euclidean of the vector $\textbf{h}$ is defined as $||\textbf{h}|| = \sqrt{\textbf{h}^T\textbf{h}} = \sqrt{\sum_{r = 1}^{n}h_r^2}$. The infinity norm (or maximum norm) of the vector $\textbf{h}$ is $||\textbf{h}||_{\infty} = \max\{|h_r|: r \in [1,n]\}$.
Let $[x]_+ = max\{0,x\} = \left\{\begin{matrix} x & \textrm{if } x \ge 0,\\ 0 & \textrm{otherwise} \end{matrix}\right.$ where $x \in \mathbb{R}$. We define the operator $[\textbf{x}]_+$ for a vector $\textbf{x} = \textrm{col}\bigl\{x_r: 1 \le r \le n\bigr\}$ as the following equation \[[\textbf{x}]_+ = \textrm{col}\bigl\{[x_r]_+: 1 \le r \le n\bigr\}.\]

Denote by $Expt[X]$ and $Var[X]$ the expectation and variance of a random variable $X$, respectively.

Let $\phi: \mathcal{X} \rightarrow \mathbb{R}$ be a function where $\mathcal{X} \subseteq \mathbb{R}^n$ is its domain.
We use $\overline{\eth_{\phi}}(\textbf{x})$ to denote the set of subgradients of the function $\phi$ at the given point $\textbf{x}$.
The function $\phi$ is convex if there exists a constant $\sigma_{\phi} \ge 0$ satisfying that, for any $\partial \phi(\textbf{x}) \in \overline{\eth_{\phi}}(\textbf{x})$, $\partial \phi(\textbf{y}) \in \overline{\eth_{\phi}}(\textbf{y})$, we have
\begin{equation}\label{eq_convex_def}
(\partial \phi(\textbf{x}) - \partial \phi(\textbf{y}))^T(\textbf{x} - \textbf{y}) \ge \sigma_{\phi}||\textbf{x} - \textbf{y}||^2
\end{equation}
The parameter $\sigma_{\phi}$ is called the convexity modulus.
If $\sigma_{\phi} > 0$, then $\phi$ is called a strongly convex function.
When the function $\phi$ is continuously differentiable, it has the unique gradient $\nabla \phi(\textbf{u}) \in \overline{\eth_{\phi}}(\textbf{u})$.
The function $\phi$ is called a smooth function if it has a Lipschitz gradient constant $L_{\phi}$ such that
\begin{equation}\label{eq_smooth_def}
||\nabla\phi(\textbf{x}) - \nabla\phi(\textbf{x})|| \le L_{\phi}||\textbf{x} -  \textbf{y}||, \forall \textbf{x}, \textbf{y} \in \mathcal{X}.
\end{equation}
%%%%%%%%%%%%%%%%%%%%%%%%%%%%%%%%%%%%%%%%%%%%%%%%%%%%%%%%%%%%%%%%%%%%%%%%%%%%%%%%%%%%%%%%%%
%%%%%%%%%%%%%%%%%%%%%%%%%%%%%%%%%%%%%%%%%%%%%%%%%%%%%%%%%%%%%%%%%%%%%%%%%%%%%%%%%%%%%%%%%%
\section{MPC energy management problem}
%%%%%%%%%%%%%%%%%%%%%%%%%%%%%%%%%%%%%%%%%%%%%%%%%%%%%%%%%%%%%%%%%%%%%%%%%%%%%%%%%%%%%%%%%%
\subsection{Network of interconnected microgrids}
%%%%%%%%%%%%%%%%%%%%%%%%%%%%%%%%%%%%%%%%%%%%%%%%%%%%%%%%%%%%%%%%%%%%%%%%%%%%%%%%%%%%%%%%%%
Consider a network of $N$ interconnected microgrids where each of them consists of power consumers and can transmit power to some others. In addition, some of these microgrids have either or both power generators and energy storage units.
We use a graph $\mathcal{G} = (\mathcal{V}, \mathcal{E})$ to represent the network where ${\mathcal{V}} = \{1, 2, \dots, N\}$ is the node set and $\mathcal{E}$ is the edge set.
Each node in the set $\mathcal{V}$ represents a microgrid and an edge $(i,j) \in \mathcal{E} \subset \mathcal{V} \times \mathcal{V}$ describes a power transmission line between two microgrids $i$ and $j$.
Fig. \ref{fig_interconnected_microgrids} illustrates a network of $5$ interconnected microgrids.
In this paper, we assume that the represented graph $\mathcal{G}$ is undirected and connected. That means the fact $(i,j) \in \mathcal{E}$ implies $(j,i) \in \mathcal{E}$, i.e., power can be transmitted from the microgrid $i$ to the microgrid $j$, and vice versa. In addition, for any pair $i,j \in \mathcal{V}$, there exists at least one path $k_1, k_2, \dots, k_r$ where $k_1 = i, k_r = j$ and $(k_{m},k_{m+1}) \in {\mathcal{E}}, \forall 1 \le m \le r-1$.
\begin{Assumption}\label{aspt_1}
$\mathcal{G}$ is an undirected and connected graph.
\end{Assumption}
\begin{figure}
\begin{center}
\includegraphics[width=0.45\textwidth, angle=0]{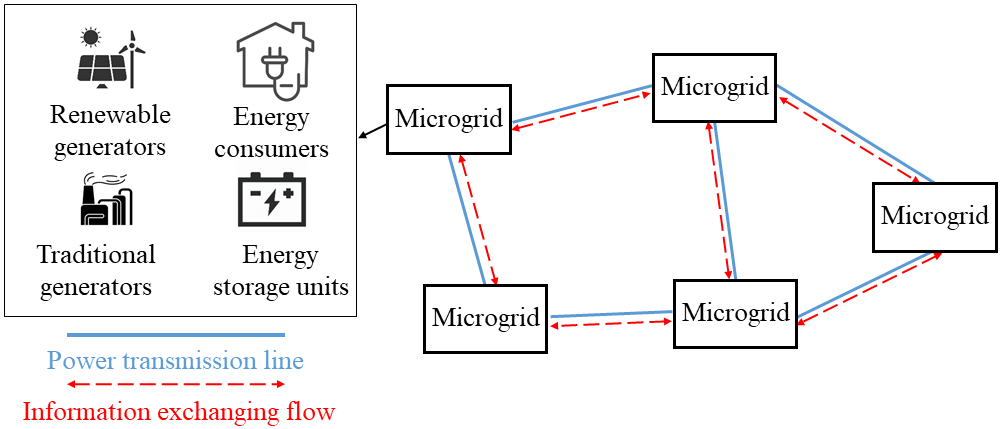}
\caption{A network of $5$ interconnected microgrids.}\label{fig_interconnected_microgrids}
\end{center}
\end{figure}
Moreover, we denote by $\mathcal{V}^G \subset \mathcal{V}$ the set of all microgrids which have power (traditional/renewable) generators and denote by $\mathcal{V}^E \subset \mathcal{V}$ the set of all microgrids which have energy storage units. In general, $\mathcal{V}^G \cap \mathcal{V}^E \neq \emptyset$.
Let $P_i(t)$ be the total power generated by power generators in the microgrid $i \in \mathcal{V}^G$ for the time interval $[tT, (t+1)T]$ where $t$ is the time index, and $T$ is the sampling time. 
Due to the limitation of energy resources, the generated power of the microgrid $i \in \mathcal{V}^G$ is restricted in a feasible range. As the capacities of power generators using renewable energy resources, such as solar and wind, depend highly on unpredictable conditions (e.g., weather), their feasible ranges may vary. Denote by $\overline{P_i}(t)$ the upper bound capacity of power generators in the microgrid $i \in \mathcal{V}^G$ in the time interval $[tT, (t+1)T]$.
We have
\begin{equation}\label{eq_generation_range}
0 \le P_i(t) \le \overline{P_i}(t), \forall i \in \mathcal{V}^G.
\end{equation}
In addition, the change of the generated powers of microgrids are restricted by the following ramp constraints.
\begin{equation}\label{eq_generation_ramp}
\underline{R_i} \le P_i(t) - P_i(t-1) \le \overline{R_i}, \forall i \in \mathcal{V}^G,
\end{equation}
where $\underline{R_i}$ and $\overline{R_i}$ are bounds of the ramp rates in the microgrid $i \in \mathcal{V}^G$, respectively.
%For convenience, we set conventionally $P_i(t) = \overline{P_i}(t) = \underline{R_i} = \overline{R_i} = 0, \forall i \notin V^G$, for all $t \ge 0$.
Let $E_i^c(t)$ and $E_i^d(t)$ be the power delivered to or by the energy storage unit of the microgrid $i \in \mathcal{V}^E$.
The dynamics of the energy storage unit in microgrid $i \in \mathcal{V}^E$ is described as
\begin{equation}\label{eq_storage_unit}
SoC_i(t+1) = (1 - eff_i)SoC_i(t) + c_iE_i^c(t) - d_iE_i^d(t)
\end{equation}
where $SoC_i(t)$ is the state-of-charge at time $tT$ of the storage unit in the microgrid $i$, and $eff_i \in [0,1), c_i > 0, d_i > 0$ are given parameters. The limitations of capacity and operation of every energy storage unit are given by following inequalities
\begin{subequations}\label{eq_storage_unit_ineq}
\begin{gather}
0 \le SoC_i(t) \le \overline{SoC_i}, \forall i \in \mathcal{V}^E, t \ge 0,\\
0 \le E_i^c(t) \le \overline{E_i^c}, \forall i \in \mathcal{V}^E, t \ge 0,\\
0 \le E_i^d(t) \le \overline{E_i^d}, \forall i \in \mathcal{V}^E, t \ge 0,
\end{gather}
\end{subequations}
where $\overline{SoC_i}, \overline{E_i^c}, \overline{E_i^d}$ are known upper bounds.

Concerning the environment and economic problem, renewable power generators have become more desired. However, they usually have small capacities and highly affected by the intermittent nature of renewable energy resources.
Then, for some microgrid without traditional power generators, using only their own power generators and energy storage units may not serve enough their power demands. Moreover, some microgrids have no power generators or energy storage units.
So, it is necessary that microgrids exchange energy via power transmission lines.
Denote by $F_{ij}(t)$ the power transmission from the microgrid $i$ to the microgrid $j$ in the time interval $[tT, (t+1)T]$.
Depending on the material and the length of the transmission line $(i,j) \in \mathcal{E}$, its transmitted power is restricted in a limited range as \eqref{eq_transmission_range}.
\begin{equation}\label{eq_transmission_range}
0 \le F_{ij}(t) \le \overline{F_{ij}}, \forall (i,j) \in \mathcal{E}.
\end{equation}
Let $\mathcal{N}_i = \{j \in {\mathcal{V}}: (i,j) \in \mathcal{E}\}$ be the set of neighbors of the microgrid $i \in \mathcal{V}$. It is clear that power can be transmitted from the microgrid $i$ to the microgrid $j$ only if $j \in \mathcal{N}_i$.
%%%%%%%%%%%%%%%%%%%%%%%%%%%%%%%%%%%%%%%%%%%%%%%%%%%%%%%%%%%%%%%%%%%%%%%%%%%%%%%%%%%%%%%%%%
\subsection{MPC approach}
%%%%%%%%%%%%%%%%%%%%%%%%%%%%%%%%%%%%%%%%%%%%%%%%%%%%%%%%%%%%%%%%%%%%%%%%%%%%%%%%%%%%%%%%%%
Let $D_i(t)$ be the power demand of the microgrid $i \in \mathcal{V}$ in $[tT, (t+1)T]$.
Ignoring the power transmission loss, from the conservation law, the power level of the node $i \in \mathcal{V}$ is
\begin{equation}\label{eq_energy_level}
\scalebox{0.825}{$\Delta P_i(t) = P_i(t) - E_i^c(t) + E_i^d(t) + \sum\limits_{j \in {\mathcal{N}}_i}\left[F_{ji}(t) - F_{ij}(t)\right] - D_i(t)$}
\end{equation}
For easy formulation, we consider here that $P_i(t) = 0, \forall i \notin \mathcal{V}^G$, and $E_i^c(t) = E_i^d(t) = 0, \forall i \notin \mathcal{V}^E$ are constant for all $t \ge 0$.
The requirement of avoiding the lack of power consuming in each microgrid can be stated by the following inequality:
\begin{equation}\label{eq_satisfied_demand}
\Delta P_i(t) \ge 0, \forall i \in \mathcal{V}, t \ge 0.
\end{equation}
Due to physical constraints (\ref{eq_generation_range}-\ref{eq_transmission_range}), there are some cases where the requirement \eqref{eq_satisfied_demand} can not be satisfied. For example, the total power demands are significantly larger than the summation of current generated powers and states-of-charge in energy storage.
To overcome these risks, MPC energy management strategies focus on finding an optimal control plan over a finite-time horizon $\{t_0+1, t_0+2, \dots, t_0+K\}$ for the network of interconnected microgrids under the assumption that the power demands can be predicted over this time horizon. Let $t_0$ be the current time, and $K$ be the chosen horizon length. With the power demands prediction, a model predictive control (MPC) problem for the coordination of the network of interconnected microgrids has the following form:
\begin{align*}
{\mathcal{P}}(t_0): & \min\limits_{P_i(t), F_{ij}(t), E_i^c(t), E_i^d(t)} \Phi(t_0)\\
&\textrm{s.t. } (\ref{eq_generation_range}-\ref{eq_satisfied_demand}), \forall t_0 \le t \le t_0 + K - 1.
\end{align*}
where the cost function $\Phi(t_0)$ represents the total cost for power generation, power transmission, and storage operation.
By solving ${\mathcal{P}}(t_0)$, the obtained optimal solution corresponds to the feasible control decisions for the microgrids in next $K$-steps that have the minimum cost.
In order to avoid myopic control decisions, the horizon length is set $K > 1$ and only the control decisions corresponding to the current time, i.e., $P_i(t_0), \forall i \in \mathcal{V}^G, E_i^c(t_0), E_i^d(t_0), \forall i \in \mathcal{V}^E,$ and $F_{ij}(t_0), \forall (i,j) \in \mathcal{E}$, are applied. The overall process is repeated in every step with the newly updated states and predicted power demands.

In this paper, we consider the cost function:
\begin{equation}\label{eq_total_cost}
\Phi(t_0) = \sum\limits_{t = t_0}^{t_0+K-1}\left(\sum\limits_{i \in \mathcal{V}^G}\Phi_i^g(t) + \sum\limits_{i \in \mathcal{V}^E}\Phi_i^e(t) + \sum\limits_{(i,j) \in \mathcal{E}}\Phi_{ij}^f(t)\right)
\end{equation}
where $\Phi_i^g(t) = \Phi_i^g(P_i(t))$ is the cost to generate the power $P_i(t)$ of power generators in the microgrid $i \in \mathcal{V}^G$, $\Phi_i^e(t) = \Phi_i^c(E_i^c(t)) + \Phi_i^d(E_i^d(t))$ represents the cost for changing state-of-charge in the energy storage unit in the microgrid $i \in \mathcal{V}$, and $\Phi_{ij}^f(t) = \Phi_{ij}^f(F_{ij}(t))$ is the cost for transfer the power $F_{ij}(t)$ from the microgrid $i$ to the microgrid $j$, in the time interval $[tT,(t+1)T]$.
The cost functions $\Phi_i^g(\cdot), \forall i \in \mathcal{V}^G$, $\Phi_i^c(\cdot), \Phi_i^d(\cdot), \forall i \in \mathcal{V}$, and $\Phi_{ij}^f(\cdot), \forall (i,j) \in \mathcal{E},$ are usually assumed to have quadratic forms, e.g., in \cite{GuillermoIrisarri1998, XSHan2001, DiWu2017, HyoSungAhn2018}.
In this paper, we make the following assumption:
\begin{Assumption}\label{aspt_cost}
All generation cost functions $\Phi_i^g(\cdot), \forall i \in \mathcal{V}^G$, operation cost functions of energy storage units $\Phi_i^c(\cdot), \Phi_i^d(\cdot), \forall i \in \mathcal{V}$, and transmission cost functions $\Phi_{ij}^t(\cdot), \forall (i,j) \in \mathcal{E},$ are strongly convex and smooth. Moreover, these functions are increasing in $[0, \infty)$.
\end{Assumption}
\begin{Remark}
In practical, the power exchanging between two microgrids $i,j$, where $(i,j), (j,i) \in \mathcal{E}$, should be transmitted in one direction, i.e., either from $i$ to $j$ or vice versa. It requires that $F_{ij}(t)F_{ji}(t) = 0, \forall(i,j) \in \mathcal{E}$, $t \ge 0$. Similarly, one energy storage unit should be only charged or discharged in one time interval, i.e., $E_i^c(t)E_i^d(t) = 0, \forall i \in \mathcal{V}^E, t \ge 0$.
In this paper, we do not include these constraints in the energy management problem since they are guaranteed at the optimal control decisions (as shown later in Lemma \ref{lm_optimal_decisions}).
\end{Remark}
%%%%%%%%%%%%%%%%%%%%%%%%%%%%%%%%%%%%%%%%%%%%%%%%%%%%%%%%%%%%%%%%%%%%%%%%%%%%%%%%%%%%%%%%%%
%%%%%%%%%%%%%%%%%%%%%%%%%%%%%%%%%%%%%%%%%%%%%%%%%%%%%%%%%%%%%%%%%%%%%%%%%%%%%%%%%%%%%%%%%%
\section{Problem statement}
%%%%%%%%%%%%%%%%%%%%%%%%%%%%%%%%%%%%%%%%%%%%%%%%%%%%%%%%%%%%%%%%%%%%%%%%%%%%%%%%%%%%%%%%%%
\subsection{Reformulation for uncertain prediction}
%%%%%%%%%%%%%%%%%%%%%%%%%%%%%%%%%%%%%%%%%%%%%%%%%%%%%%%%%%%%%%%%%%%%%%%%%%%%%%%%%%%%%%%%%%
One challenging issue in formulating MPC-based energy management $\mathcal{P}(t_0)$ is the uncertainties in predicting the power demands and the power capacities, i.e., $D_i(t) \forall i \in \mathcal{V}$ and $\overline{P_i}(t) \forall i \in \mathcal{V}^G$ for all $t \ge t_0$.
It is because the fluctuation in the demands of the energy consumers and the intermittent nature of the resources.
Fortunately, based on historical data, the expectations and variations of power demands and generation capacities can be estimated with high confidence in a short term.
So, this paper considers the energy management problem under the following assumption.
\begin{Assumption}\label{aspt_3}
Let $t_0$ be the current time index. Expectations $Expt[\overline{P_i}(t)] \forall i \in \mathcal{V}^G,$ $Expt[D_i(t)] \forall i \in \mathcal{V}$, and variations $Var[\overline{P_i}(t)] \forall i \in \mathcal{V}^G, Var[D_i(t)] \forall i \in \mathcal{V}$ are given for all $t = t_0, t_0+1, \dots, t_0+K-1$.
\end{Assumption}

As the power demands are predicted with uncertainties, it is difficult to guarantee the inequality \eqref{eq_satisfied_demand} perfectly. It needs to be changed into the following chance constraint:
\begin{equation}\label{eq_constraint_prob_demand}
{\mathbb{P}}[\Delta P_i(t) \ge D_i(t)] \ge 1 - \epsilon, \forall i \in \mathcal{V}.
\end{equation}
In \eqref{eq_constraint_prob_demand}, $\mathbb{P}[\cdot]$ represents the probability of the event in $[\cdot]$ and $\epsilon$ is a given parameter.
As stated in \cite{Calafiore2006}, the constraint \eqref{eq_constraint_prob_demand} can be converted into the following inequality
\begin{equation}\label{eq_chacneConstraintsIneqBalance}
P_i(t) - E_i^c(t) + E_i^d(t) + \sum\limits_{j \in {\mathcal{N}}_i}\left[F_{ji}(t) - F_{ij}(t)\right] \ge \tilde{D}_{i,t}, \forall i \in \mathcal{V}
\end{equation}
where $\tilde{D}_{i,t} = Expt[D_i(t)] + \sqrt{\frac{1 - \epsilon}{\epsilon}}Var[D_i(t)], \forall t \in [t_0,t_0+K+1], i \in \mathcal{V}$.
Similarly, the constraint \eqref{eq_generation_range} is converted into
\begin{equation}\label{eq_chacneConstraintsGeneration}
0 \le P_i(t) \le \tilde{P}_{i,t}, \forall i \in \mathcal{V}^G
\end{equation}
where $\tilde{P}_{i,t} = Expt[\overline{P_i}(t)] - \sqrt{\frac{1 - \epsilon}{\epsilon}}Var[\overline{P_i}(t)], \forall t \in [t_0,t_0+K+1], i \in \mathcal{V}$.
Replacing the uncertain constraints (\ref{eq_energy_level}-\ref{eq_satisfied_demand}) and \eqref{eq_generation_range} by \eqref{eq_constraint_prob_demand} and \eqref{eq_chacneConstraintsIneqBalance}, respectively, we obtain the stochastic MPC energy management problem $\hat{\mathcal{P}}(t_0)$ as follows.
\begin{align*}
\hat{\mathcal{P}}(t_0): & \min\limits_{P_i(t), F_{ij}(t), E_i^c(t), E_i^d(t)} \Phi(t_0)\\
\textrm{s.t. }& \eqref{eq_chacneConstraintsGeneration}, \eqref{eq_generation_ramp}, \eqref{eq_storage_unit}, \eqref{eq_storage_unit_ineq}, \eqref{eq_transmission_range}, \eqref{eq_chacneConstraintsIneqBalance}, \forall t_0 \le t \le t_0 + K - 1.
\end{align*}
It is clear that $\hat{\mathcal{P}}(t_0)$ is a strongly convex optimization problem. Then this problem has a unique optimal solution if its feasible set is nonempty. So, we make the following assumption.
\begin{Assumption}\label{aspt_4}
For all $t_0$, there exists at least one set of vectors $\{\textbf{P}(t_0), \textbf{E}^c(t_0), \textbf{E}^d(t_0), \textbf{F}(t_0)\}$ satisfying all constraints in $\hat{\mathcal{P}}(t_0)$ strictly. In which, $\textbf{P}(t_0) = \textrm{col}\{P_i(t_0+k): i \in \mathcal{V}^G, k \in [0,K-1]\}$, $\textbf{E}^c(t_0) = \textrm{col}\{E_i^c(t_0+k): i \in \mathcal{V}^E, k \in [0,K-1]\}$, $\textbf{E}^d(t_0) = \textrm{col}\{E_i^d(t_0+k): i \in \mathcal{V}^E, k \in [0,K-1]\}$, and $\textbf{F}(t_0) = \textrm{col}\{F_{ij}(t_0+k): i \in \mathcal{E}, k \in [0,K-1]\}$.
\end{Assumption}
%%%%%%%%%%%%%%%%%%%%%%%%%%%%%%%%%%%%%%%%%%%%%%%%%%%%%%%%%%%%%%%%%%%%%%%%%%%%%%%%%%%%%%%%%%
\subsection{Communication graph}
%%%%%%%%%%%%%%%%%%%%%%%%%%%%%%%%%%%%%%%%%%%%%%%%%%%%%%%%%%%%%%%%%%%%%%%%%%%%%%%%%%%%%%%%%%
When the power network consists of many interconnected microgrids, the stochastic MPC problem $\hat{\mathcal{P}}(t_0)$ may have a large size. Then centralized control strategies, which use only one computational unit to solve $\hat{\mathcal{P}}(t_0)$, can be ineffective.
Instead, distributed approaches can achieve better scalability and robustness since each microgrid $i \in \mathcal{V}$ is assigned to a local controller.
This controller is in charge of implementing the optimal control decisions and collecting information (e.g., current states of local power generators, energy storage unit, and predicted power demands) corresponding to its microgrid. By abuse of notation, we use the agent $i$ to refer the local controller of the microgrid $i \in \mathcal{V}$. When $j \in \mathcal{N}_i$, the agent $j$ is called a neighbor of the agent $i$.
To determine the optimal control decisions, agents cooperate to solve the stochastic MPC problem $\hat{\mathcal{P}}(t_0)$.
As the computational load is shared among agents, the distributed setup reduces the computation complexities of optimization methods and can achieve the optimal solution with small computation time.

Assuming that the local control decisions of each agent $i \in \mathcal{V}$ consists of the generated powers, i.e., $\hat{P}_i(t_0+k), \forall k \in [0,K-1]$, (if $i \in \mathcal{V}^G$), the powers delivered to and by the energy storage unit, i.e., $\hat{E}_i^c(t_0+k), \hat{E}_i^d(t_0+k), \forall k \in [0,K-1]$, and the power transmitted from the microgrid $i$ to its neighbors, i.e., $\hat{F}_{ij,k}, \forall j \in \mathcal{N}_i, k \in [0,K-1]$.
From the constraints in \eqref{eq_chacneConstraintsIneqBalance}, the power transmission in the line $(i,j) \in \mathcal{E}$ needs to be determined by the coordination of two agents $i,j$.
So, we make the following assumption for the information network of interconnect microgrids.
\begin{Assumption}\label{as_communicate}
For every agent $i \in \mathcal{V}$, it is able to communicate with the agent $j$ if and only if $j \in \mathcal{N}_i$.
\end{Assumption}

We denote the optimal solution of the problem $\hat{\mathcal{P}}(t_0)$ by variables with the superscript $opt$, i.e., $P_i^{opt}(t_0+k)$ for all $i \in \mathcal{V}^G$, $E_i^{c,opt}(t_0+k), E_i^{d,opt}(t_0+k)$ for all $i \in \mathcal{V}^E$, $F_{ij}^{opt}(t_0+k) \forall j \in \mathcal{N}_i$ for all $i \in \mathcal{V}$.
The following lemma guarantees that, in the optimal solution of the problem $\hat{\mathcal{P}}(t_0)$, there is at most one direction of power transmitted in each edge and each energy storage unit is not charged and discharged at the same time.
\begin{Lemma}\label{lm_optimal_decisions}
If Assumption \ref{aspt_4} is satisfied, we have
\begin{subequations}\label{eq_properties}
\begin{gather}
F_{ij}^{opt}(t_0+k)F_{ji}^{opt}(t_0+k) = 0, \forall (i,j) \in \mathcal{E}, k \in [0,K-1],\\
E_{i}^{c,opt}(t_0+k)E_{i}^{d,opt}(t_0+k) = 0, \forall i \in \mathcal{V}^E, k \in [0,K-1].
\end{gather}
\end{subequations}
\end{Lemma}
\begin{proof}
Assume that there exist some $(i,j) \in \mathcal{E}$ and $k \in [0, K-1]$ where $F_{ij}^{opt}(t_0+k)$ and $F_{ji}^{opt}(t_0+k)$ are both positive and $F_{ij}^{opt}(t_0+k) \ge F_{ji}^{opt}(t_0+k)$ without loss of generality. It is easy to verify that changing $F_{ij}^{opt}(t_0+k), F_{ji}^{opt}(t_0+k)$ with $F_{ij}^{opt}(t_0+k) - F_{ji}^{opt}(t_0+k), 0$ still guarantees the feasibility but reduces the cost function of the problem $\hat{\mathcal{P}}(t_0)$. This contradicts  with the definition of the optimal solution. Thus, $F_{ij}^{opt}(t_0+k)F_{ji}^{opt}(t_0+k) = 0, \forall (i,j) \in \mathcal{E}, k \in [0, K-1]$.
Similarly, we can prove that $E_{i}^{c,opt}(t_0+k)E_{i}^{d,opt}(t_0+k) = 0, \forall i \in \mathcal{V}, k \in [0, K-1]$.
\end{proof}

To conclude this section, we state the main problem of this paper as follows:
\begin{Problem}\label{problem_statement}
Design a distributed method for every agent $i \in \mathcal{V}$ to find its optimal control decisions while using only its own information and communicating with its neighbors in $\mathcal{N}_i$. The optimal control decisions belonging to the agent $i$ consists of $P_i^{opt}(t_0+k)$ (if $i \in \mathcal{V}^G$), $E_i^{c,opt}(t_0+k), E_i^{d,opt}(t_0+k)$ (if $i \in \mathcal{V}^E$) and $F_{ij}^{opt}(t_0+k) \forall j \in \mathcal{N}_i$ for all $k \in [0, K-1]$.
\end{Problem}
%%%%%%%%%%%%%%%%%%%%%%%%%%%%%%%%%%%%%%%%%%%%%%%%%%%%%%%%%%%%%%%%%%%%%%%%%%%%%%%%%%%%%%%%%%
%%%%%%%%%%%%%%%%%%%%%%%%%%%%%%%%%%%%%%%%%%%%%%%%%%%%%%%%%%%%%%%%%%%%%%%%%%%%%%%%%%%%%%%%%%
\section{Distributed solution based on dual ascent method}
%%%%%%%%%%%%%%%%%%%%%%%%%%%%%%%%%%%%%%%%%%%%%%%%%%%%%%%%%%%%%%%%%%%%%%%%%%%%%%%%%%%%%%%%%%
%%%%%%%%%%%%%%%%%%%%%%%%%%%%%%%%%%%%%%%%%%%%%%%%%%%%%%%%%%%%%%%%%%%%%%%%%%%%%%%%%%%%%%%%%%
\subsection{Compact form of MPC energy management problem}
%%%%%%%%%%%%%%%%%%%%%%%%%%%%%%%%%%%%%%%%%%%%%%%%%%%%%%%%%%%%%%%%%%%%%%%%%%%%%%%%%%%%%%%%%%
Define vectors: $\hat{\textbf{P}}_i = \textrm{col}\{P_i(t_0+k): k \in [0,K-1]\}$ for all $i \in \mathcal{V}^G$, $\hat{\textbf{E}}_i = \textrm{col}\{\textrm{col}\{E_{c,i}(t_0+k): k \in [0,K-1]\}, \textrm{col}\{E_{d,i}(t_0+k): k \in [0,K-1]\}\}$ for all $i \in \mathcal{V}^E$, and $\hat{\textbf{F}}_i = \textrm{col}\{\textrm{col}\{F_{ij}(t_0+k): j \in \mathcal{N}_i\}: k \in [0,K-1]\}$ for all $i \in \mathcal{V}$. Then the vector of all local variables of each agent $i$ is
\begin{equation}\label{eq_local_variables_vector}
\hat{\textbf{x}}_i = \left\{\begin{matrix}
\textrm{col}\{\hat{\textbf{P}}_i, \hat{\textbf{E}}_i, \hat{\textbf{F}}_i\}, & \textrm{if } i \in \mathcal{V}^G \cap \mathcal{V}^E,\\
\textrm{col}\{\hat{\textbf{P}}_i, \hat{\textbf{F}}_i\}, & \textrm{if } i \in \mathcal{V}^G \backslash \mathcal{V}^E,\\
\textrm{col}\{\hat{\textbf{E}}_i, \hat{\textbf{F}}_i\}, & \textrm{if } i \in \mathcal{V}^E \backslash \mathcal{V}^G,\\
\hat{\textbf{F}}_i, & \textrm{otherwise}.
\end{matrix}\right.
\end{equation}
Let $\Phi_i(\hat{\textbf{x}}_i)$ be the local cost function of the agent $i$ which is the sum of all cost functions depending local variables of the agent $i$. We have:
\begin{equation}\label{eq_local_costfunction}
\Phi_i(\hat{\textbf{x}}_i) = \left\{\begin{matrix}
\sum_{k = 0}^{K-1}\Bigl(\Phi_i^g(P_i(t_0+k)) + \Phi_i^c(E_{c,i}(t_0+k)) + \Phi_i^d(E_{d,i}(t_0+k)) + \sum_{j \in \mathcal{N}_i} \Phi_{ij}^f(F_{ij}(t_0+k))\Bigr), & \textrm{if } i \in \mathcal{V}^G \cap \mathcal{V}^E,\\
\sum_{k = 0}^{K-1}\Bigl(\Phi_i^g(P_i(t_0+k)) + \sum_{j \in \mathcal{N}_i} \Phi_{ij}^f(F_{ij}(t_0+k))\Bigr), & \textrm{if } i \in \mathcal{V}^G \backslash \mathcal{V}^E,\\
\sum_{k = 0}^{K-1}\Bigl(\Phi_i^c(E_{c,i}(t_0+k)) + \Phi_i^d(E_{d,i}(t_0+k)) + \sum_{j \in \mathcal{N}_i} \Phi_{ij}^f(F_{ij}(t_0+k))\Bigr), & \textrm{if } i \in \mathcal{V}^E \backslash \mathcal{V}^G,\\
\sum_{k = 0}^{K-1} \sum_{j \in \mathcal{N}_i} \Phi_{ij}^f(F_{ij}(t_0+k)), & \textrm{otherwise}.
\end{matrix}\right..
\end{equation}
Denote by $n_i$ the dimension of the local variable vector $\hat{\textbf{x}}_i$ for every agent $i \in \mathcal{V}$. It is clear that the local cost function $\Phi_i(\hat{\textbf{x}}_i)$ is the sum of $n_i$ cost functions, each of them depends on one variable $[\hat{\textbf{x}}_i]_r$ where $ 1 \le r \le n_i$.
For easy notation, we define the function $\Phi_{ir}([\hat{\textbf{x}}_i]_r)$ such that $\Phi_i(\hat{\textbf{x}}_i) = \sum_{r = 1}^{n_i} \Phi_{ir}([\hat{\textbf{x}}_i]_r)$.
Under Assumption \ref{aspt_cost}, the function $\Phi_{ir}([\hat{\textbf{x}}_i]_r)$ is strongly convex and smooth, $\forall r \in [1,n_i], i \in \mathcal{V}$. Define $\sigma_{ir}$ and $L_{ir}$ as the convexity modulus and the Lipschitz gradient constant of the function $\Phi_{ir}([\hat{\textbf{x}}_i]_r)$. We have $\sigma_{ir} > 0$ and $L_{ir} > 0$, $\forall r \in [1,n_i], i \in \mathcal{V}$.
Define
\begin{subequations}\label{eq_parameter_cost}
\begin{align}
\sigma_{min} &= \min\limits_{i \in \mathcal{V}} \left\{\min\limits_{r \in [1,n_i]} \{\sigma_{ir}\}\right\},\\
L_{max} &= \max\limits_{i \in \mathcal{V}} \left\{\max\limits_{r \in [1,n_i]} \{L_{ir}\}\right\}.
\end{align}
\end{subequations}
As definition, we have $\Phi(t_0) = \sum_{i \in \mathcal{V}} \sum_{r = 1}^{n_i} \Phi_{ir}([\hat{\textbf{x}}_i]_r)$ and $\Phi_{ir}([\hat{\textbf{x}}_i]_r)$ is a function of one variable for all $r \in [1, n_i], \forall i \in \mathcal{V}$. So, $\Phi(t_0)$ is a strongly convex function with modulus $\sigma_{min} > 0$ and is a smooth function having Lipschitz gradient constant $L_{max}$.

In the stochastic MPC energy management problem $\hat{\mathcal{P}}(t_0)$, the inequalities in \eqref{eq_chacneConstraintsGeneration}, \eqref{eq_generation_ramp}, \eqref{eq_storage_unit}, \eqref{eq_storage_unit_ineq}, \eqref{eq_transmission_range} are local constraints since they depend on local information of individual agents.
Define the matrix $\textbf{A}_i^g$ and the vector $\textbf{a}_i^g$ for each agent $i \in \mathcal{V}^G$ such that the constraints \eqref{eq_chacneConstraintsGeneration}, \eqref{eq_generation_ramp} can be described in the following matrix form:
\begin{equation}\label{eq_matrixform_temp1}
\textbf{A}_i^g\hat{\textbf{x}}_i \le \textbf{a}_i^g, \forall i \in \mathcal{V}^G.
\end{equation}
Similarly, we can find matrices $\textbf{A}_i^e, \textbf{A}_i^t$, and the vectors $\textbf{a}_i^e, \textbf{a}_i^t$ such that the local constraints \eqref{eq_storage_unit}, \eqref{eq_storage_unit_ineq} and \eqref{eq_transmission_range} can be described in the matrix forms \eqref{eq_matrixform_temp2} and \eqref{eq_matrixform_temp3}, respectively.
\begin{equation}\label{eq_matrixform_temp2}
\textbf{A}_i^e\hat{\textbf{x}}_i \le \textbf{a}_i^e, \forall i \in \mathcal{V}^E.
\end{equation}
\begin{equation}\label{eq_matrixform_temp3}
\textbf{A}_i^t\hat{\textbf{x}}_i \le \textbf{a}_i^t, \forall i \in \mathcal{V}.
\end{equation}
Also we define the matrices $\textbf{A}_i, \textbf{B}_{ii}, \textbf{B}_{ij} \forall j \in \mathcal{N}_i$, and the vectors $\textbf{a}_i, \textbf{b}_i$ such that
$\textbf{A}_i = \textrm{blkcol}\{\textbf{A}_i^g, \textbf{A}_i^e, \textbf{A}_i^t\}, \textbf{a}_i = \textrm{col}\{\textbf{a}_i^g, \textbf{a}_i^e, \textbf{a}_i^t\}$, $\textbf{B}_{ii}\hat{\textbf{x}}_i = \textrm{col}\{P_i(t) + E_i^d(t) - E_i^c(t) - \sum_{j \in \mathcal{N}_i}F_{ij}(t)|t \in [t_0, t_0+K-1]\}$ for $i \in \mathcal{V}^G \cap \mathcal{V}^E$; $\textbf{A}_i = \textrm{blkcol}\{\textbf{A}_i^e, \textbf{A}_i^t\}, \textbf{a}_i = \textrm{col}\{\textbf{a}_i^e, \textbf{a}_i^t\}$, $\textbf{B}_{ii}\hat{\textbf{x}}_i = \textrm{col}\{E_i^d(t) - E_i^c(t) - \sum_{j \in \mathcal{N}_i}F_{ij}(t)|t \in [t_0, t_0+K-1]\}$ for $i \in \mathcal{V}^G \backslash \mathcal{V}^E$; $\textbf{A}_i = \textrm{blkcol}\{\textbf{A}_i^g, \textbf{A}_i^t\}, \textbf{a}_i = \textrm{col}\{\textbf{a}_i^g, \textbf{a}_i^t\}$, $\textbf{B}_{ii}\hat{\textbf{x}}_i = \textrm{col}\{P_i(t) - \sum_{j \in \mathcal{N}_i}F_{ij}(t)|t \in [t_0, t_0+K-1]\}$ for $i \in \mathcal{V}^E \backslash \mathcal{V}^G$; $\textbf{A}_i = \textbf{A}_i^t, \textbf{a}_i = \textbf{a}_i^t$, $\textbf{B}_{ii}\hat{\textbf{x}}_i = \textrm{col}\{-\sum_{j \in \mathcal{N}_i}F_{ij}(t)|t \in [t_0, t_0+K-1]\}$ for $i \notin \mathcal{V}^G \cup \mathcal{V}^E$; and $\textbf{B}_{ij}\hat{\textbf{x}}_i = \textrm{col}\{F_{ij}(t)|t \in [t_0, t_0+K-1]\}$ for all $j \in \mathcal{N}_i, i \notin \mathcal{V}$.
Then the problem $\hat{\mathcal{P}}(t_0)$ can be rewritten in matrix form as \eqref{eq_MPC_matrixform} as follows.
\begin{subequations}\label{eq_MPC_matrixform}
\begin{align}
\min\limits_{\hat{\textbf{x}}_i, \forall i \in \mathcal{V}}& \sum\limits_{i \in \mathcal{V}} \Phi_i(\hat{\textbf{x}}_i)\\
\textrm{s.t. }& \textbf{A}_i\hat{\textbf{x}}_i \le \textbf{a}_i, \forall i \in \mathcal{V},\\
& \textbf{B}_{ii}\hat{\textbf{x}}_i + \sum\limits_{j \in \mathcal{N}_i}\textbf{B}_{ji}\hat{\textbf{x}}_j \le \textbf{b}_i, \forall i \in \mathcal{V}.
\end{align}
\end{subequations}
It is easy to see that $\textbf{B}_{ij}\hat{\textbf{x}}_i = \textrm{col}\{-F_{ij}(t):t \in [t_0, t_0+K-1]\}, \forall j \in \mathcal{N}_i$, and $\textbf{b}_i = \textrm{col}\{-\tilde{D}_{i,t}: t \in [t_0, t_0+K-1]\}$ for all $i \in \mathcal{V}$. The detailed formulations of $\textbf{A}_i, \textbf{a}_i$ and $\textbf{B}_{ii}$ are given in Appendix.
For each agent $i$, we use $m_i$ to denote the number of rows in the local matrix $\textbf{A}_i$, i.e., $\textbf{A}_i \in \mathbb{R}^{m_i \times n_i}$. It is clear that $\textbf{B}_{ii}, \textbf{B}_{ij} \in \mathbb{R}^{K \times n_i}$.
%%%%%%%%%%%%%%%%%%%%%%%%%%%%%%%%%%%%%%%%%%%%%%%%%%%%%%%%%%%%%%%%%%%%%%%%%%%%%%%%%%%%%%%%%%
\subsection{Distributed solution method}
%%%%%%%%%%%%%%%%%%%%%%%%%%%%%%%%%%%%%%%%%%%%%%%%%%%%%%%%%%%%%%%%%%%%%%%%%%%%%%%%%%%%%%%%%%
Let $\hat{\boldsymbol{\lambda}}_i, \forall i \in \mathcal{V}$, and $\hat{\boldsymbol{\zeta}}_i, \forall i \in \mathcal{V}$, be the dual variables corresponding to the inequality constraints (\ref{eq_MPC_matrixform}b), and (\ref{eq_MPC_matrixform}c), respectively. The Lagrangian function of the problem \eqref{eq_MPC_matrixform} is given as $\mathcal{L}(\hat{\textbf{x}}, \hat{\boldsymbol{\eta}}) = \sum\limits_{i \in \mathcal{V}} \Bigl\{\Phi_i(\hat{\textbf{x}}_i) + \hat{\boldsymbol{\lambda}}_i^T(\textbf{A}_i\hat{\textbf{x}}_i - \textbf{a}_i) + \hat{\boldsymbol{\zeta}}_i^T(\textbf{B}_{ii}\hat{\textbf{x}}_i + \sum\limits_{j \in \mathcal{N}_i}\textbf{B}_{ji}\hat{\textbf{x}}_j - \textbf{b}_i)\Bigr\}$.
Here, we use $\hat{\textbf{x}} = \textrm{col}\bigl\{\hat{\textbf{x}}_i: i \in \mathcal{V}\bigr\}$ and $\hat{\boldsymbol{\eta}} = \textrm{col}\bigl\{[\hat{\boldsymbol{\lambda}}_i^T, \hat{\boldsymbol{\zeta}}_i^T]^T: i \in \mathcal{V}\bigr\}$ as stacked vectors of primal and dual variables, respectively.
Then the dual function corresponding to the problem \eqref{eq_MPC_matrixform} is given as
\begin{equation}\label{eq_dual_function_MPC}
\Psi(\hat{\boldsymbol{\eta}}) = \min\limits_{\hat{\textbf{x}}} \mathcal{L}(\hat{\textbf{x}}, \hat{\boldsymbol{\eta}}).
\end{equation}
From the definition, the dual function is a concave function.
We can verify that $\mathcal{L}(\hat{\textbf{x}}, \hat{\boldsymbol{\eta}}) = \sum_{i \in \mathcal{V}}\mathcal{L}_i(\hat{\textbf{x}}_i, \hat{\boldsymbol{\eta}})$ where
\begin{align}
\mathcal{L}_i(\hat{\textbf{x}}_i, \hat{\boldsymbol{\eta}}) = \Phi_i(\hat{\textbf{x}}_i) + \Bigl(\textbf{A}_i^T\hat{\boldsymbol{\lambda}}_i + \textbf{B}_{ii}^T\hat{\boldsymbol{\zeta}}_i + \sum_{j \in \mathcal{N}_i}\textbf{B}_{ij}^T\hat{\boldsymbol{\zeta}}_j\Bigr)^T\hat{\textbf{x}}_i - \textbf{a}_i^T\hat{\boldsymbol{\lambda}}_i - \textbf{b}_i^T\hat{\boldsymbol{\zeta}}_i.\label{eq_local_LagrangianFunction}
\end{align}
Define the vector function $\textbf{x}^*(\hat{\boldsymbol{\eta}}) = \textrm{col}\{\textbf{x}_i^*(\hat{\boldsymbol{\eta}}): i \in \mathcal{V}\}$ where
\begin{equation}\label{eq_sup_optimalsolution}
\textbf{x}_i^*(\hat{\boldsymbol{\eta}}) = \arg\min\limits_{\hat{\textbf{x}}_i} \mathcal{L}_i(\hat{\textbf{x}}_i, \hat{\boldsymbol{\eta}})
\end{equation}
Due to the separation of the Lagrangian function, we have $\textbf{x}^*(\hat{\boldsymbol{\eta}}) = \arg\min_{\hat{\textbf{x}}} \mathcal{L}(\hat{\textbf{x}}, \hat{\boldsymbol{\eta}})$.
It is easy to see that the function $\mathcal{L}_i(\hat{\textbf{x}}_i, \hat{\boldsymbol{\eta}})$ is a strongly conve function with respect to the variable $\hat{\textbf{x}}_i$. Then, for each fixed dual variable vector $\hat{\boldsymbol{\eta}}$, the vector $\textbf{x}_i^*(\hat{\boldsymbol{\eta}})$ in \eqref{eq_sup_optimalsolution} is unique.
According to Proposition \ref{prop_gradientproject}, the dual function $\Psi(\hat{\boldsymbol{\eta}})$ has the following property:
\begin{Lemma}\label{lm_grad_dualfunction}
$\Psi(\hat{\boldsymbol{\eta}})$ is continuously differentiable with
\begin{subequations}\label{eq_dual_gradient}
\begin{align}
\nabla_{\hat{\boldsymbol{\lambda}}_i} \Psi(\hat{\boldsymbol{\eta}}) &= \textbf{A}_i\textbf{x}_i^*(\hat{\boldsymbol{\eta}}) - \textbf{a}_i, \forall i \in \mathcal{V},\\
\nabla_{\hat{\boldsymbol{\zeta}}_i} \Psi(\hat{\boldsymbol{\eta}}) &= \textbf{B}_{ii}\textbf{x}_i^*(\hat{\boldsymbol{\eta}}) + \sum_{j \in \mathcal{N}_i}\textbf{B}_{ji}\textbf{x}_j^*(\hat{\boldsymbol{\eta}}) - \textbf{b}_i, \forall i \in \mathcal{V}.
\end{align}
\end{subequations}
\end{Lemma}

Since the dual function $\Psi(\hat{\boldsymbol{\eta}})$ is a concave function, the function $-\Psi(\hat{\boldsymbol{\eta}})$ is convex. In addition, $-\Psi(\hat{\boldsymbol{\eta}})$ is a continuous differentiable function. Then, one optimal solution $\boldsymbol{\eta}^{opt} = \arg\min_{\hat{\boldsymbol{\eta}} \ge \textbf{0}} \{-\Psi(\hat{\boldsymbol{\eta}})\}$ can be found by applying a gradient-based method. It is clear that $\boldsymbol{\eta}^{opt}$ is also the optimal solution to the following problem
\begin{equation}\label{eq_dual_MPC}
\max\limits_{\hat{\boldsymbol{\eta}} \ge \textbf{0}} \textrm{ } \Psi(\hat{\boldsymbol{\eta}}).
\end{equation}
The optimization problem \eqref{eq_dual_MPC} is the dual problem corresponding to the stochastic MPC energy management problem \eqref{eq_MPC_matrixform}. Then $\boldsymbol{\eta}^{opt}$ is one optimal dual solution of the problem \eqref{eq_MPC_matrixform}. Applying the gradient projection method \cite{DimitriPBertsekas1999} for solving the dual problem \eqref{eq_dual_MPC}, we derive the following iteration method
\begin{align}\label{eq_grad_MPC}
\boldsymbol{\eta}(s+1) &= \arg\min\limits_{\hat{\boldsymbol{\eta}} \ge \textbf{0}} \left|\left|\hat{\boldsymbol{\eta}} - \boldsymbol{\eta}(s) - \frac{1}{L_{\Psi}}\nabla \Psi(\boldsymbol{\eta}(s))\right|\right|^2\nonumber\\
&= \left[\boldsymbol{\eta}(s) + \frac{1}{L_{\Psi}}\nabla \Psi(\boldsymbol{\eta}(s))\right]_+.
\end{align}
where $L_{\Psi}$ is the Lipschitz gradient constant of the dual function $\Psi(\hat{\boldsymbol{\eta}})$.

Define $\textbf{x}_i^*(s)$ as the minimizer of the function $\mathcal{L}_i(\hat{\textbf{x}}_i,\tilde{\boldsymbol{\eta}}(s))$ for the estimated dual solution $\tilde{\boldsymbol{\eta}}(s)$, i.e., $\textbf{x}_i^*(s) = \textbf{x}_i^*(\hat{\boldsymbol{\eta}}(s))$.
Then the detailed formulation of the update \eqref{eq_grad_MPC} is rewritten in detailed forms as follows.
\begin{subequations}\label{eq_grad_MPCdetailed}
\begin{align}
\textbf{x}_i^*(s) &= \arg\min\limits_{\hat{\textbf{x}}_i} \left\{\Phi_i(\hat{\textbf{x}}_i) + \Bigl(\textbf{A}_i^T\boldsymbol{\lambda}_i(s) + \textbf{B}_{ii}^T\boldsymbol{\zeta}_i(s) + \sum_{j \in \mathcal{N}_i}\textbf{B}_{ij}^T\boldsymbol{\zeta}_j(s)\Bigr)^T\hat{\textbf{x}}_i\right\},\\
\boldsymbol{\lambda}_i(s+1) &= \left[\boldsymbol{\lambda}_i(s) + \frac{1}{L_{\Psi}}\Bigl(\textbf{A}_i\textbf{x}_i^*(s) - \textbf{a}_i\Bigr)\right]_+.\\
\boldsymbol{\zeta}_i(s+1) &= \left[\boldsymbol{\zeta}_i(s) + \frac{1}{L_{\Psi}}\Bigl(\textbf{B}_{ii}\textbf{x}_i^*(s) + \sum_{j \in \mathcal{N}_i}\textbf{B}_{ji}\textbf{x}_j^*(s) - \textbf{b}_i\Bigr)\right]_+,
\end{align}
\end{subequations}
for all $i \in \mathcal{V}$.
As the update \eqref{eq_grad_MPCdetailed} is applying the gradient projection method to solve the dual problem \eqref{eq_dual_MPC}, we have the following theorem
\begin{Theorem}
The update \eqref{eq_grad_MPCdetailed} guarantees the asymptotic convergence of the estimated dual solution $\boldsymbol{\eta}(s)$ to one optimal dual solution $\boldsymbol{\eta}^{opt}$ of the problem \eqref{eq_dual_MPC}.
\begin{equation}
\lim\limits_{s \rightarrow \infty} \boldsymbol{\eta}(s) = \boldsymbol{\eta}^{opt} = \arg\max\limits_{\hat{\boldsymbol{\eta}} \ge \textbf{0}} \textrm{ } \Psi(\hat{\boldsymbol{\eta}}).
\end{equation}
\end{Theorem}

When the optimal dual solution of the stochastic MPC based energy management problem \eqref{eq_MPC_matrixform} is found, every agent $i$ can determine its corresponding optimal solution as
\begin{equation}\label{eq_MPC_optimalsolution}
\textbf{x}_i^*(\boldsymbol{\eta}^{opt}) = \arg\min\limits_{\hat{\textbf{x}}_i} \left\{\Phi_i(\hat{\textbf{x}}_i) + \Bigl(\textbf{A}_i^T\boldsymbol{\lambda}_i^{opt} + \textbf{B}_{ii}^T\boldsymbol{\zeta}_i^{opt} + \sum_{j \in \mathcal{N}_i}\textbf{B}_{ij}^T\boldsymbol{\zeta}_j^{opt}\Bigr)^T\hat{\textbf{x}}_i\right\}, \forall i \in \mathcal{V}.
\end{equation}
According to Proposition \ref{prop_optimalsolution}, the vector $\textbf{x}^{opt} = \textrm{col}\left\{\textbf{x}_i^*(\boldsymbol{\eta}^{opt}): i \in \mathcal{V}\right\}$ is the optimal solution of the stochastic MPC based energy management problem \eqref{eq_MPC_matrixform}. Since the problem \eqref{eq_MPC_matrixform} is a strongly convex optimization problem, the optimal solution $\textbf{x}^{opt}$ is unique.
%%%%%%%%%%%%%%%%%%%%%%%%%%%%%%%%%%%%%%%%%%%%%%%%%%%%%%%%%%%%%%%%%%%%%%%%%%%%%%%%%%%%%%%%%%
%%%%%%%%%%%%%%%%%%%%%%%%%%%%%%%%%%%%%%%%%%%%%%%%%%%%%%%%%%%%%%%%%%%%%%%%%%%%%%%%%%%%%%%%%%
\section{ADMM-based solution method}
%%%%%%%%%%%%%%%%%%%%%%%%%%%%%%%%%%%%%%%%%%%%%%%%%%%%%%%%%%%%%%%%%%%%%%%%%%%%%%%%%%%%%%%%%%
\subsection{Compact form of MPC energy management problem}
%%%%%%%%%%%%%%%%%%%%%%%%%%%%%%%%%%%%%%%%%%%%%%%%%%%%%%%%%%%%%%%%%%%%%%%%%%%%%%%%%%%%%%%%%%
Consider the constraints \eqref{eq_chacneConstraintsIneqBalance} of the microgrid $i \in \mathcal{V}$, they depend not only its own local variables but also the power transmitted from its neighbors in $\mathcal{N}_i$ to this microgrid. To cope with this issue, we assume that the agent $i$ stores copies of these variables. Let $F_{ji}^{copy}(t_0+k) = F_{ji}(t_0+k), \forall j \in \mathcal{N}_i, k = 0, \dots, K-1$. The copy $F_{ji}^{copy}(t_0+k), \forall k = 0, \dots, K-1$, can substitute to the variable $F_{ji}(t_0+k)$ in the equations in \eqref{eq_chacneConstraintsIneqBalance} to convert these constraints to local constraints of the agent $i$. Thus, the vector of all local control variables in the agent $i$ is defined as $\hat{\textbf{u}}_i = col\left\{\hat{\textbf{x}}_i, col\left\{\hat{\textbf{x}}_{ij}: j \in \mathcal{N}_i\right\}\right\}$ where $\hat{\textbf{x}}_i$ is given in \eqref{eq_local_variables_vector} and
\[\hat{\textbf{x}}_{ij} = col\left\{F_{ji}^{copy}(t_0+k): j \in \mathcal{N}_i, k = 0, \dots, K-1\right\}, \forall j \in \mathcal{N}_i.\]

Consider the constraints $F_{ji}^{copy}(t_0+k) = F_{ji}(t_0+k), \forall j \in \mathcal{N}_i, k = 0, \dots, K-1$, as the coupled constraints between two agents $i$ and $j$. To describe these constraints in a stacked form, we define two matrices $\textbf{S}_{ij}$ and $\textbf{R}_{ij}$ such that $\textbf{S}_{ij}\hat{\textbf{u}}_i = \hat{\textbf{x}}_{ij}$ and $\textbf{R}_{ij}\hat{\textbf{u}}_i = col\left\{F_{ij}(t_0+k): 0 \le k \le K-1\right\}$, $\forall j \in \mathcal{N}_i$. So, for all $i = 1, \dots, N$, we have
\begin{equation}\label{eq_ADMMver_coupledequalities}
\textbf{H}_{ij}\hat{\textbf{u}}_i = \textbf{R}_{ji}\hat{\textbf{u}}_j, \forall j \in \mathcal{N}_i.
\end{equation}
Define the matrix $\textbf{M}_i = \left[\begin{matrix}-\textbf{A}_i & \textbf{0}_{n_i \times K|\mathcal{N}_i|}\\ -\textbf{B}_{ii} & \textbf{1} \otimes \textbf{I}_K\end{matrix}\right]$ and the vector $\textbf{m}_i = \left[\begin{matrix}-\textbf{a}_i\\ -\textbf{b}_i\end{matrix}\right]$. It is easy to verify that two constraints (\ref{eq_MPC_matrixform}b,c) can be rewritten as \eqref{eq_ADMMver_localineq}
\begin{equation}\label{eq_ADMMver_localineq}
\textbf{M}_i \hat{\textbf{u}}_i \ge \textbf{m}_i, \forall i \in \mathcal{V}.
\end{equation}
So, we have the problem \eqref{eq_MPC_ADMMver} which is equivalent to the stochastic MPC based energy management problem $\hat{\mathcal{P}}(t_0)$ as follows.
\begin{subequations}\label{eq_MPC_ADMMver}
\begin{align}
\min\limits_{\hat{\textbf{u}}_i, \forall i \in \mathcal{V}}& \sum\limits_{i \in \mathcal{V}} \hat{\Phi}(\hat{\textbf{u}}_i)\\
\textrm{s.t. }& \textbf{M}_i\hat{\textbf{u}}_i \le \textbf{m}_i, \forall i \in \mathcal{V},\\
& \textbf{S}_{ij}\hat{\textbf{u}}_i = \textbf{R}_{ji}\hat{\textbf{u}}_j, \forall j \in \mathcal{N}_i, \forall i \in \mathcal{V}.
\end{align}
\end{subequations}
where the local cost function $\hat{\Phi}(\hat{\textbf{u}}_i) = \Phi(\hat{\textbf{x}}_i)$ given in \eqref{eq_local_costfunction}.
%%%%%%%%%%%%%%%%%%%%%%%%%%%%%%%%%%%%%%%%%%%%%%%%%%%%%%%%%%%%%%%%%%%%%%%%%%%%%%%%%%%%%%%%%%
\subsection{Distributed solution method}
%%%%%%%%%%%%%%%%%%%%%%%%%%%%%%%%%%%%%%%%%%%%%%%%%%%%%%%%%%%%%%%%%%%%%%%%%%%%%%%%%%%%%%%%%%
To easily apply the ADMM algorithm \eqref{eq_ADMM} for solving the optimization problem \eqref{eq_MPC_ADMMver}, we rewrite this problem in the form of \eqref{eq_ADMMproblem}.
Define $\hat{\textbf{v}}_i = \textbf{M}_i\hat{\textbf{u}}_i - \textbf{m}_i$ and $\hat{\textbf{v}}_{ij}^{(1)} = \textbf{S}_{ij}\hat{\textbf{u}}_i, \hat{\textbf{v}}_{ij}^{(2)} = \textbf{R}_{ij}\hat{\textbf{u}}_i, \forall j \in \mathcal{N}_i$ and define the set $\Omega_i = \mathbb{R}_{+}^{m_i+K}$ for all local inequality constraints for agent $i$, the set $\Omega_{ij} = \left\{\textbf{v}_{ij}^{(1)}, \textbf{v}_{ji}^{(2)}: \textbf{v}_{ij}^{(1)} = \textbf{v}_{ji}^{(2)}\right\}$ for all coupled constraints between two agents $i$ and $j$.
Then we have the optimization problem \eqref{eq_SMPC_ADMMconsensus} which is equivalent to the problem \eqref{eq_MPC_ADMMver} as follows.
\begin{subequations}\label{eq_SMPC_ADMMconsensus}
\begin{align}
\min &\textrm{ } \sum\limits_{i \in \mathcal{V}}\hat{\Phi}_i(\hat{\textbf{u}}_i)\\
\textrm{s.t. }& \hat{\textbf{u}}_i \in \Omega_i, \forall i \in \mathcal{V},\\
&(\hat{\textbf{v}}_{ij}^{(1)}, \hat{\textbf{v}}_{ji}^{(2)}) \in \Omega_{ij}, \forall j \in \mathcal{N}_i, \forall i \in \mathcal{V},\\
&\textbf{M}_i\hat{\textbf{u}}_i - \hat{\textbf{v}}_i = \textbf{m}_i, \forall i \in \mathcal{V},\\
&\textbf{S}_{ij}\hat{\textbf{u}}_i - \hat{\textbf{v}}_{ij}^{(1)} = , \forall j \in \mathcal{N}_i, \forall i \in \mathcal{V},\\
&\textbf{R}_{ij}\hat{\textbf{u}}_i - \hat{\textbf{v}}_{ij}^{(2)} = , \forall j \in \mathcal{N}_i, \forall i \in \mathcal{V}.
\end{align}
\end{subequations}

Define the stacked vectors $\textbf{u} = \left[\begin{matrix} \hat{\textbf{u}}_1^T & \hat{\textbf{u}}_2^T & \cdots & \hat{\textbf{u}}_N^T\end{matrix} \right]^T$ and $\textbf{v} = \left[ \begin{matrix}\tilde{\textbf{v}}_1^T & \tilde{\textbf{v}}_2^T & \cdots & \tilde{\textbf{v}}_N^T\end{matrix} \right]^T$ where \[\tilde{\textbf{v}}_i = col\left\{\hat{\textbf{v}}_i, col\left\{ col\bigl\{\hat{\textbf{v}}_{ij}^{(1)}, \hat{\textbf{v}}_{ij}^{(2)}\bigr\}: j \in \mathcal{N}_i\right\}\right\}.\]
It is clear that the problem \eqref{eq_SMPC_ADMMconsensus} has the form of \eqref{eq_ADMMproblem} with
\[\phi_u(\textbf{u}) = \sum\limits_{i \in \mathcal{V}}\hat{\Phi}_i(\hat{\textbf{u}}_i), \phi_v(\textbf{v}) = 0,\]
$\textbf{H}_u = blkdiag\left\{\textbf{H}_1, \textbf{H}_2, \dots, \textbf{H}_N\right\}$  where
$\textbf{H}_i = blkcol\left\{\textbf{M}_i, blkcol\left\{\textbf{S}_{ij}, \textbf{R}_{ij}: j \in \mathcal{N}_i\right\}\right\}$, $\textbf{H}_v = -\textbf{I}$, $\textbf{h} = col\left\{\textbf{h}_1, \textbf{h}_2, \dots, \textbf{h}_N\right\}$ where $\textbf{h}_i = col\left\{\textbf{m}_i, \textbf{0}\right\}$, and $\Omega_u = \mathbb{R}^n$ where $n = \sum_{i \in \mathcal{V}}\left(n_i + |\mathcal{N}_i|K\right)$, $\Omega_v = \bigtimes_{i = 1}^{N} \left( \Omega_i \times \bigtimes_{j \in \mathcal{N}_i} \Omega_{ij}\right)$.

We use $\boldsymbol{\lambda}_i$, $\boldsymbol{\lambda}_{ij}^{(1)}$, and $\boldsymbol{\lambda}_{ij}^{(2)}$ to denote the dual variables corresponding to the equality constraints (\ref{eq_SMPC_ADMMconsensus}d), (\ref{eq_SMPC_ADMMconsensus}e), and (\ref{eq_SMPC_ADMMconsensus}f), respectively.
By applying the proximal ADMM algorithm \eqref{eq_ADMM} to solve the problem \eqref{eq_SMPC_ADMMconsensus}, we derive the update law \eqref{eq_updatelaw} as follows.
\begin{subequations}\label{eq_updatelaw}
\begin{align}
\textbf{v}_i(s+1) =\textrm{ }& \left[\textbf{M}_i\textbf{u}_i(s) - \textbf{m}_i - \frac{1}{\rho}\boldsymbol{\lambda}_i(s)\right]_{+},\\
\textbf{v}_{ij}^{(1)}(s+1) =\textrm{ }& \frac{1}{2}\Bigl( \textbf{S}_{ij}\textbf{u}_i(s) - \frac{1}{\rho}\boldsymbol{\lambda}_{ij}^{(1)}(s) + \textbf{R}_{ji}\textbf{u}_j(s) - \frac{1}{\rho}\boldsymbol{\lambda}_{ji}^{(2)}(s) \Bigr), \forall j \in \mathcal{N}_i,\\
\textbf{v}_{ij}^{(2)}(s+1) =\textrm{ }& \frac{1}{2}\Bigl( \textbf{S}_{ji}\textbf{u}_j(s) - \frac{1}{\rho}\boldsymbol{\lambda}_{ji}^{(1)}(s) + \textbf{R}_{ij}\textbf{u}_i(s) - \frac{1}{\rho}\boldsymbol{\lambda}_{ij}^{(2)}(s) \Bigr), \forall j \in \mathcal{N}_i,\\
\textbf{u}_i(s+1) =\textrm{ }& \arg\min\limits_{\hat{\textbf{u}}_i} \left\{\hat{\Phi}_i(\hat{\textbf{u}}_i) + \frac{\rho}{2}\left|\left|\textbf{M}_i\hat{\textbf{u}}_i - \hat{\textbf{v}}_i(s+1) - \textbf{m}_i - \frac{1}{\rho}\boldsymbol{\lambda}_i(s)\right|\right|^2\right.\nonumber\\ &\left.+ \frac{\rho}{2}\sum_{j \in \mathcal{N}_i} \left(\left|\left|\textbf{S}_{ij}\hat{\textbf{u}}_i - \hat{\textbf{v}}_{ij}^{(1)}(s+1) - \frac{1}{\rho}\boldsymbol{\lambda}_{ij}^{(1)}(s)\right|\right|^2 + \left|\left|\textbf{R}_{ij}\hat{\textbf{u}}_i - \hat{\textbf{v}}_{ij}^{(2)}(s+1) - \frac{1}{\rho}\boldsymbol{\lambda}_{ij}^{(2)}(s)\right|\right|^2 \right)\right\}\\
\boldsymbol{\lambda}_i(s+1) =\textrm{ }& \boldsymbol{\lambda}_i(s) - \rho\left(\textbf{M}_i\textbf{u}_i(s+1) - \textbf{m}_i - \textbf{v}_i(s+1)\right),\\
\boldsymbol{\lambda}_{ij}^{(1)}(s+1) =\textrm{ }& \boldsymbol{\lambda}_{ij}^{(1)}(s) - \rho\left(\textbf{S}_{ij}\textbf{u}_i(s+1) - \textbf{v}_{ij}^{(1)}(s+1)\right), \forall j \in \mathcal{N}_i,\\
\boldsymbol{\lambda}_{ij}^{(2)}(s+1) =\textrm{ }& \boldsymbol{\lambda}_{ij}^{(2)}(s) - \rho\left(\textbf{R}_{i_j}\textbf{u}_i(s+1) - \textbf{v}_{i_j}^{(2)}(s+1)\right), \forall j \in \mathcal{N}_i.
\end{align}
\end{subequations}
for all $i \in \mathcal{V}$.
The asymptotic convergence of ADMM iterations to the optimal solution of the problem \eqref{eq_SMPC_ADMMconsensus} is guaranteed by Proposition \ref{prop_ADMM}.

Let $\textbf{P}_i = \left[\begin{matrix}\textbf{I}_{n_i} & \textbf{0}_{n_i \times K|\mathcal{N}_i|}\end{matrix}\right]$, we have
\begin{equation}
\hat{\textbf{x}}_i = \textbf{P}_i \hat{\textbf{u}}_i.
\end{equation}
Moreover, we have
\begin{Theorem}\label{th_updatelaw}
The optimal solution of the distributed stochastic MPC traffic signal control problem \eqref{eq_MPC_matrixform} is achieved asymptotically by the update law \eqref{eq_updatelaw} in the following sense
\begin{equation}
\lim\limits_{s \rightarrow \infty} \textbf{P}_i \textbf{u}_i(s) = \textbf{x}_i^{opt}.
\end{equation}
\end{Theorem}
%%%%%%%%%%%%%%%%%%%%%%%%%%%%%%%%%%%%%%%%%%%%%%%%%%%%%%%%%%%%%%%%%%%%%%%%%%%%%%%%%%%%%%%%%%
%%%%%%%%%%%%%%%%%%%%%%%%%%%%%%%%%%%%%%%%%%%%%%%%%%%%%%%%%%%%%%%%%%%%%%%%%%%%%%%%%%%%%%%%%%
\section{Numerical simulations}
%%%%%%%%%%%%%%%%%%%%%%%%%%%%%%%%%%%%%%%%%%%%%%%%%%%%%%%%%%%%%%%%%%%%%%%%%%%%%%%%%%%%%%%%%%
In this section, we test the effectiveness of our proposed method in solving the stochastic MPC-based energy management problem \eqref{eq_MPC_matrixform}. The tested power network is a modified IEEE 30-bus system as shown in Fig. \ref{fig_test_system}.
\begin{figure}[htb]
\begin{center}
\includegraphics[width=0.325\textwidth, angle=0]{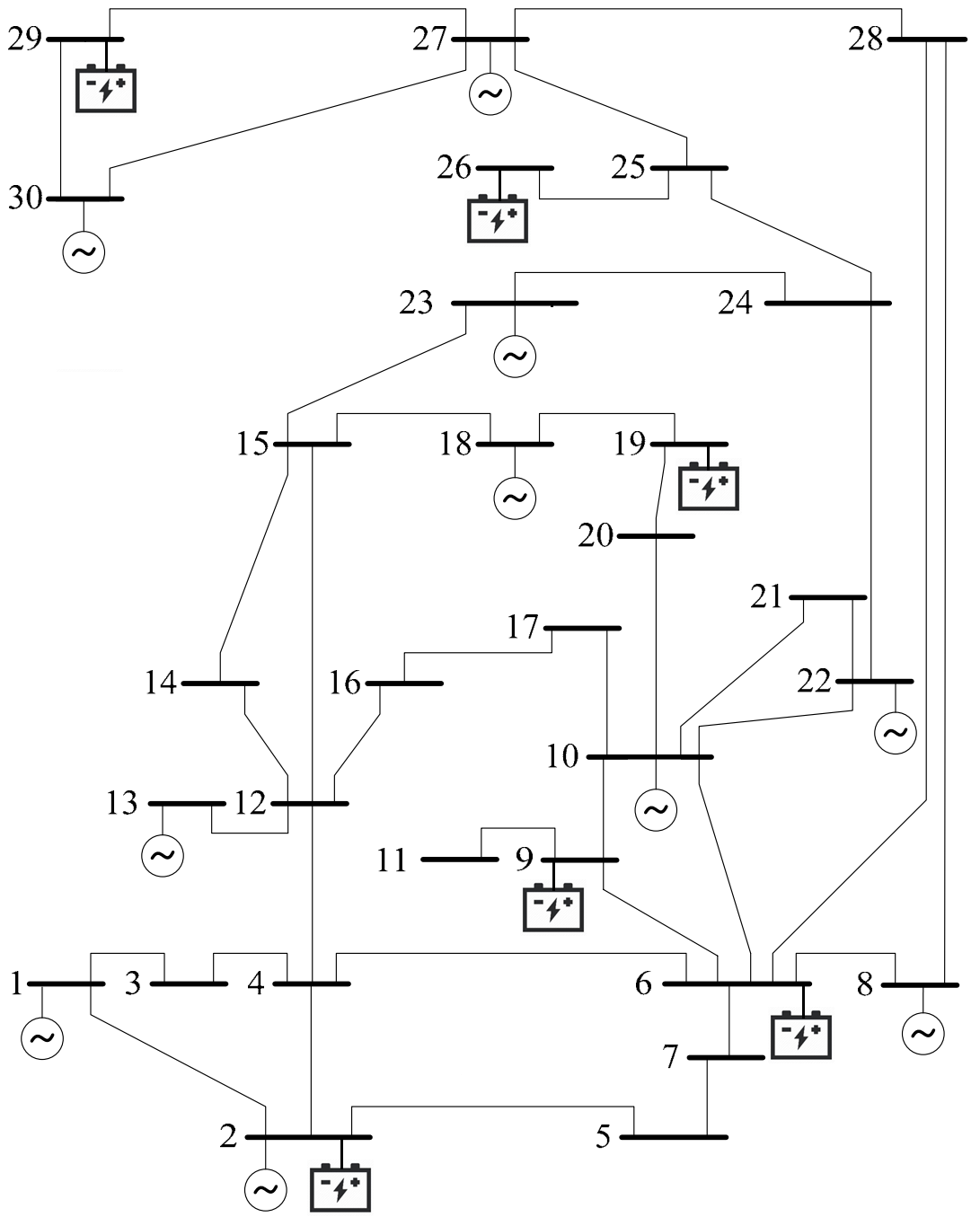}
\caption{Test system with 10 generators.}\label{fig_test_system}
\end{center}
\end{figure}
\begin{table}
\begin{center}
\captionof{table}{Data of power generators.}\label{tbl_generator}
\scalebox{0.75}
{\begin{tabular}{c|c|c|c|c|c|c|c}
\hline\rule{0pt}{12pt} 
Unit&  Bus&  {$\alpha_i$}&  {$\omega_i$}&  {$\overline{P_i}$}&  {$\underline{R_i}$}&  {$\overline{R_i}$} \\[0.5ex] \hline\rule{0pt}{9pt}
G1&  1&  2.0&  110.25&  400 & -80 & 80\\
G2&  2&  1.25&  140.75&  500 & -100 & 100\\
G3&  8&  1.2&  14.00&  1000 & -300 & 300\\
G4&  10&  1.25&  140.00&  500 & -100 & 100\\
G5&  13&  1.0&  100.25&  600 & -120 & 120\\
G6&  18&  1.4&  120.50&  400 & -80 & 80\\
G7&  22&  0.5&  150.60&  1000 & -200 & 200\\
G8&  23&  0.9&  170.00&  800 & -160 & 160\\
G9&  27&  1.6&  130.25&  300 & -80 & 80\\
G10&  30&  0.8&  110.25&  1000 & -150 & 150\\\hline
\end{tabular}}
\end{center}
\end{table}
\begin{table}
\begin{center}
\captionof{table}{Data of energy storage units.}\label{tbl_store}
\scalebox{0.75}
{\begin{tabular}{c|c|c|c|c|c|c|c|c}
\hline\rule{0pt}{12pt}
Unit&  Bus&  {$eff_i$}&  {$c_i$}&  {$d_i$}&  {$\overline{E_i^d}=\overline{E_i^d}$}&  {$\overline{SoC_i}$}&  {$\upsilon_i$}&  {$\varsigma_i$} \\ [0.5ex]\hline\rule{0pt}{9pt}
E1&  2&  0.95&  0.92&  0.94 & 50 & 500 & 0.75 & 0.85\\
E2&  6&  0.98&  0.95&  0.95 & 100 & 1000 & 1.2 & 1.1\\
E3&  10&  0.92&  0.91&  0.96 & 80 & 500 & 1.25 & 1.5\\
E4&  19&  0.94&  0.93&  0.97 & 60 & 800 & 1.1 & 1.0\\
E5&  26&  0.99&  0.94&  0.98 & 50 & 500 & 1.25 & 1.15\\
E6&  29&  0.96&  0.91&  0.92 & 100 & 900 & 1.4 & 1.2\\\hline
\end{tabular}}
\end{center}
\end{table}
\begin{table}
\begin{center}
\captionof{table}{Basic power demand.}\label{tbl_predict}
\scalebox{0.75}
{\begin{tabular}{c|c|c|c|c|c|c|c|c|c|c|c|c|c|c|c|c|c|c}
\hline\rule{0pt}{12pt} 
$k$&  $0$&  $1$&  $2$&  $3$&  $4$&  $5$&  $6$&  $7$&  $8$&  $9$&  $10$&  $11$ \\[0.5ex] \hline\rule{0pt}{9pt}
$D(t_0+k)$&  $50$&  $60$&  $70$&  $80$&  $90$&  $100$&  $100$&  $95$&  $85$&  $80$&  $70$&  $60$\\\hline
\end{tabular}}
\end{center}
\end{table}
The set of buses connecting to power generators is $\mathcal{V}^G = \{1, 2, 8, 10, 13, 18, 22, 23, 27, 30\}$ and the set of buses having energy storage units is $\mathcal{V}^E = \{2, 6, 10, 19, 26, 29\}$. Assume that all cost functions have quadratic forms as follows.
\[\Phi_i^g(P_i) = \frac{\alpha_i}{2} P_i^2 + \omega_i P_i, \forall i \in \mathcal{V}^G,\] \[\Phi_i^e(P_i) = \frac{\upsilon_i}{2} (E_i^c)^2 + \frac{\varsigma_i}{2} (E_i^d)^2, \forall i \in \mathcal{V}^E,\]
\[\Phi_{ij}^f(F_{ij}) = \frac{\varpi_{ij}}{2}F_{ij}^2 + \nu_{ij}F_{ij}\]
Parameters for power generators and energy storage units are given in TABLE. \ref{tbl_generator} and TABLE. \ref{tbl_store}, respectively. For each edge $(i,j) \in \mathcal{E}$, the parameters are chosen randomly where $\varpi_{ij} \in [0.1, 0.3], \nu_{ij} \in [0.5, 1.0], \forall (i,j) \in \mathcal{E}$. We also set randomly the upper bound $\overline{F_{ij}} \in [100, 400]$.
The expectation of the power demand of each agent $i \in \mathcal{V}$ is given by \[E\left[D_i(t_0+k)\right] = D(t_0+k) + 10*[(i+k) \textrm{ mod } 5]\] where $D(t_0+k)$ is given in TABLE. \ref{tbl_predict}. In addition, we assume that $Var\left[D_i(t_0+k)\right] = 10, \forall i \in \mathcal{V}, k \ge 0$. Let the time interval $T = 1$ hour and the parameter $\epsilon = 0.25$ for the chance constraints.
We use MATLAB to implement the dual ascent method \eqref{eq_grad_MPCdetailed} and the ADMM-based method \eqref{eq_updatelaw}. From the simulated model, we compute a Lipschitz gradient constant of the dual function $\Psi(\boldsymbol{\eta})$ as $L_{\Psi} = 100$. Then we choose the penalty parameter $\rho = 0.01$ for the ADMM-based method.
Our experiments are conducted in a computer with chip Intel Core I5 8500 and 16 GB RAM.

Fig. \ref{fig_test_simulation} illustrates the convergence of the estimated solutions under two distributed methods (the dual ascent method \eqref{eq_grad_MPCdetailed} and ADMM-based method \eqref{eq_updatelaw}) to the optimal solution $\textbf{x}^{opt}$ of the stochastic MPC-based energy management problem \eqref{eq_MPC_matrixform} with $K =3$ and $K = 6$. For the the dual ascent method \eqref{eq_grad_MPCdetailed}, the estimated solution is $\textbf{x}^*(s) = \textrm{col}\left\{\textbf{x}_i^*(\boldsymbol{\eta}(s)): i \in \mathcal{V}\right\}$ where $\boldsymbol{\eta}(s)$ is the estimated dual solution at the iteration $s \ge 1$. For the ADMM-based method \eqref{eq_updatelaw}), $\textbf{x}(s) = \textrm{col}\left\{\textbf{P}_i\textbf{u}_i(s): i \in \mathcal{V}\right\}$ where $\textbf{u}_i(s)$ is the estimated solution in (\ref{eq_updatelaw}d).
\begin{figure*}
\begin{center}
\includegraphics[width=0.48\textwidth, angle=0]{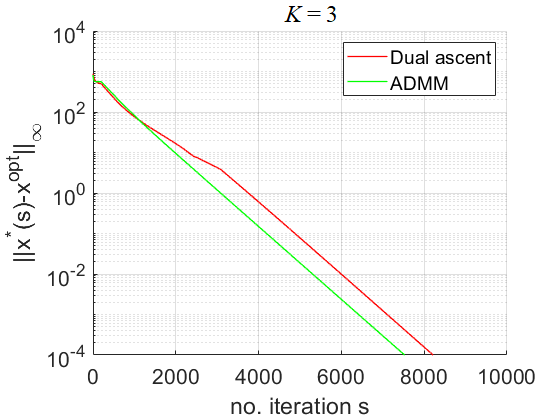}
\includegraphics[width=0.48\textwidth, angle=0]{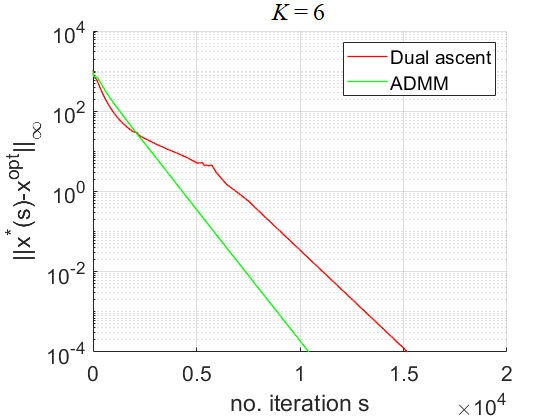}
\caption{Convergences of the estimated solutions to the precise optimal solutions under three distributed optimization methods.}\label{fig_test_simulation}
\end{center}
\end{figure*}
It is easy to observe that the ADMM-based method \eqref{eq_updatelaw}) requires less iterations than the dual ascent method \eqref{eq_grad_MPCdetailed}. However, the duration time to run one iteration of the ADMM-based method \eqref{eq_updatelaw}) is about two times longer than the dual ascent method \eqref{eq_grad_MPCdetailed}.
%%%%%%%%%%%%%%%%%%%%%%%%%%%%%%%%%%%%%%%%%%%%%%%%%%%%%%%%%%%%%%%%%%%%%%%%%%%%%%%%%%%%%%%%%%
%%%%%%%%%%%%%%%%%%%%%%%%%%%%%%%%%%%%%%%%%%%%%%%%%%%%%%%%%%%%%%%%%%%%%%%%%%%%%%%%%%%%%%%%%%
\section{Conclusion}
%%%%%%%%%%%%%%%%%%%%%%%%%%%%%%%%%%%%%%%%%%%%%%%%%%%%%%%%%%%%%%%%%%%%%%%%%%%%%%%%%%%%%%%%%%
In this paper, we design two distributed method for each agent to determine the optimal control decisions for its power distribution (if the power node has generator capacity), the operation of its energy storage unit (if existed), and the power transmitted to its neighbor. Under dual decomposition technique and ADMM algorithm, the asymptotic convergence of both two designed distributed methods is guaranteed. Numerical simulation results in MATLAB verifies the effectiveness of these methods.
%%%%%%%%%%%%%%%%%%%%%%%%%%%%%%%%%%%%%%%%%%%%%%%%%%%%%%%%%%%%%%%%%%%%%%%%%%%%%%%%%%%%%%%%%%
%%%%%%%%%%%%%%%%%%%%%%%%%%%%%%%%%%%%%%%%%%%%%%%%%%%%%%%%%%%%%%%%%%%%%%%%%%%%%%%%%%%%%%%%%%
\appendix
\subsection{Detailed form of matrices}
%%%%%%%%%%%%%%%%%%%%%%%%%%%%%%%%%%%%%%%%%%%%%%%%%%%%%%%%%%%%%%%%%%%%%%%%%%%%%%%%%%%%%%%%%%
\begin{subequations}\label{eq_matrixform_temp1_detailed}
\begin{align}
\textbf{A}_i^g &= \left[\begin{matrix} \textbf{I}_K \\ -\textbf{I}_K\\ \left[\begin{matrix} -\textbf{I}_{K-1} & \textbf{0}_{K-1} \end{matrix}\right] + \left[\begin{matrix} \textbf{0}_{K-1} & \textbf{I}_{K-1} \end{matrix}\right]\\ \left[\begin{matrix} \textbf{I}_{K-1} & \textbf{0}_{K-1} \end{matrix}\right] - \left[\begin{matrix} \textbf{0}_{K-1} & \textbf{I}_{K-1} \end{matrix}\right]
\end{matrix}\right] = \left[\begin{matrix}
1 & 0 & 0 & \cdots & 0 & 0\\ 0 & 1 & 0 & \cdots & 0 & 0\\ \vdots & \vdots & \vdots & \ddots & \vdots & \vdots\\ 0 & 0 & 0 & \cdots & 1 & 0\\ 0 & 0 & 0 & \cdots & 0 & 1\\
-1 & 0 & 0 & \cdots & 0 & 0\\ 0 & -1 & 0 & \cdots & 0 & 0\\ \vdots & \vdots & \vdots & \ddots & \vdots & \vdots\\ 0 & 0 & 0 & \cdots & -1 & 0\\ 0 & 0 & 0 & \cdots & 0 & -1\\
-1 & 1 & 0 & \cdots & 0 & 0\\ 0 & -1 & 1 & \cdots & 0 & 0\\ \vdots & \vdots & \ddots & \ddots & \vdots & \vdots\\ 0 & 0 & 0 & \cdots & 1 & 0\\ 0 & 0 & 0 & \cdots & -1 & 1\\
1 & -1 & 0 & \cdots & 0 & 0\\ 0 & 1 & -1 & \cdots & 0 & 0\\ \vdots & \vdots & \ddots & \ddots & \vdots & \vdots\\ 0 & 0 & 0 & \cdots & -1 & 0\\ 0 & 0 & 0 & \cdots & 1 & -1
\end{matrix}\right],\\
\textbf{a}_i^g &= \left[\begin{matrix}
\min\left\{\tilde{P}_{i,t_0}, P_i(t_0-1) + \overline{R_i}\right\}\\ \tilde{P}_{i,t_0+1}\\ \vdots\\ \tilde{P}_{i,t_0+K-2}\\ \tilde{P}_{i,t_0+K-1}\\
\max\left\{0, P_i(t_0-1) + \underline{R_i}\right\}\\ \textbf{0}_K\\
\overline{R_i}\textbf{1}_K\\
-\underline{R_i}\textbf{1}_K
\end{matrix}\right].
\end{align}
\end{subequations}
\begin{subequations}\label{eq_matrixform_temp2_detailed}
\begin{align}
\textbf{A}_i^e &= \left[\begin{matrix} \left[\begin{matrix} \textbf{I}_{K} & \textbf{0}_{K \times K}\\ \textbf{0}_{K \times K} & \textbf{I}_{K} \end{matrix}\right]\\
\left[\begin{matrix} -\textbf{I}_{K} & \textbf{0}_{K \times K}\\ \textbf{0}_{K \times K} & -\textbf{I}_{K} \end{matrix}\right]\\
\left[\begin{matrix} 1 & 0 & \cdots & 0\\ 0 & 1 + (1 - eff_i) & \cdots & 0\\ \vdots & \vdots & \ddots & \vdots\\ 0 & 0 & \cdots & 1 + \sum\limits_{k = 1}^{K-1}(1 - eff_i)^k \end{matrix}\right] \left[\begin{matrix} c_i\textbf{I}_{K} & -d_i\textbf{I}_{K}\\ -c_i\textbf{I}_{K} & d_i\textbf{I}_{K} \end{matrix}\right]
\end{matrix}\right]\nonumber\\ &= \scalebox{0.85}{$\left[\begin{matrix}
1 & 0 & \cdots & 0 & 0 & 0 & \cdots & 0\\ 0 & 1 & \cdots & 0 & 0 & 0 & \cdots & 0\\ \vdots & \vdots & \ddots & \vdots & \vdots & \vdots & \ddots & \vdots\\ 0 & 0 & \cdots & 1 & 0 & 0 & \cdots & 0\\
0 & 0 & \cdots & 0 & 1 & 0 & \cdots & 0\\ 0 & 0 & \cdots & 0 & 0 & 1 & \cdots & 0\\ \vdots & \vdots & \ddots & \vdots & \vdots & \vdots & \ddots & \vdots\\ 0 & 0 & \cdots & 0 & 0 & 0 & \cdots & 1\\
-1 & 0 & \cdots & 0 & 0 & 0 & \cdots & 0\\ 0 & -1 & \cdots & 0 & 0 & 0 & \cdots & 0\\ \vdots & \vdots & \ddots & \vdots & \vdots & \vdots & \ddots & \vdots\\ 0 & 0 & \cdots & -1 & 0 & 0 & \cdots & 0\\
0 & 0 & \cdots & 0 & -1 & 0 & \cdots & 0\\ 0 & 0 & \cdots & 0 & 0 & -1 & \cdots & 0\\ \vdots & \vdots & \ddots & \vdots & \vdots & \vdots & \ddots & \vdots\\ 0 & 0 & \cdots & 0 & 0 & 0 & \cdots & -1\\
c_i & 0 & \cdots & 0 & -d_i & 0 & \cdots & 0\\ 0 & c_i\left(1 + (1 - eff_i)\right) & \cdots & 0 & 0 & -d_i\left(1 + (1 - eff_i)\right) & \cdots & 0\\ \vdots & \vdots & \ddots & \vdots & \vdots & \vdots & \ddots & \vdots\\ 0 & 0 & \cdots & c_i\left(1 + \sum\limits_{k = 1}^{K-1}(1 - eff_i)^k\right) & 0 & 0 & \cdots & -d_i\left(1 + \sum\limits_{k = 1}^{K-1}(1 - eff_i)^k\right)\\
-c_i & 0 & \cdots & 0 & d_i & 0 & \cdots & 0\\ 0 & -c_i\left(1 + (1 - eff_i)\right) & \cdots & 0 & 0 & d_i\left(1 + (1 - eff_i)\right) & \cdots & 0\\ \vdots & \vdots & \ddots & \vdots & \vdots & \vdots & \ddots & \vdots\\ 0 & 0 & \cdots & -c_i\left(1 + \sum\limits_{k = 1}^{K-1}(1 - eff_i)^k\right) & 0 & 0 & \cdots & d_i\left(1 + \sum\limits_{k = 1}^{K-1}(1 - eff_i)^k\right)
\end{matrix}\right]$},\\
\textbf{a}_i^e &= \left[\begin{matrix}
\overline{E_i^c}\textbf{1}_K\\ \overline{E_i^d}\textbf{1}_K\\
\textbf{0}_{2K}\\
\overline{SoC_i} - (1 - eff_i)SoC_i(t_0-1)\\ \overline{SoC_i} - (1 - eff_i)^2 SoC_i(t_0-1)\\ \vdots\\ \overline{SoC_i} - (1 - eff_i)^{K} SoC_i(t_0-1)\\
- (1 - eff_i)SoC_i(t_0-1)\\ - (1 - eff_i)^2 SoC_i(t_0-1)\\ \vdots\\ - (1 - eff_i)^K SoC_i(t_0-1)
\end{matrix}\right].
\end{align}
\end{subequations}
\begin{equation}\label{eq_matrixform_temp3_detailed}
\textbf{A}_i^f = \left[\begin{matrix} \textbf{I}_K \otimes \textbf{I}_{|\mathcal{N}_i|} \\ -\textbf{I}_K \otimes \textbf{I}_{|\mathcal{N}_i|}\end{matrix}\right],
\textbf{a}_i^f = \left[\begin{matrix}
\textbf{1}_K \otimes \textrm{col}\{\overline{F_{ij}}: j \in \mathcal{N}_i\}\\
\textbf{0}_{K|\mathcal{N}_i|}
\end{matrix}\right].
\end{equation}
\begin{subequations}\label{eq_matrixform_temp4_detailed}
\begin{align}
\textbf{B}_{ii} = \left\{\begin{matrix}
\left[\begin{matrix} -\textbf{I}_K & \textbf{I}_K & -\textbf{I}_K & \textbf{I}_K \otimes \textbf{1}_{|\mathcal{N}_i|}^T\end{matrix}\right], & \textrm{if } i \in \mathcal{V}^G \cap \mathcal{V}^E,\\
\left[\begin{matrix} \textbf{I}_K & \textbf{I}_K \otimes \textbf{1}_{|\mathcal{N}_i|}^T\end{matrix}\right], & \textrm{if } i \in \mathcal{V}^G \backslash \mathcal{V}^E,\\
\left[\begin{matrix} \textbf{I}_K & -\textbf{I}_K & \textbf{I}_K \otimes \textbf{1}_{|\mathcal{N}_i|}^T\end{matrix}\right], & \textrm{if } i \in \mathcal{V}^E \backslash \mathcal{V}^G,\\
\textbf{I}_K \otimes \textbf{1}_{|\mathcal{N}_i|}^T, & \textrm{otherwise}.
\end{matrix}\right.
\end{align}
\end{subequations}
%%%%%%%%%%%%%%%%%%%%%%%%%%%%%%%%%%%%%%%%%%%%%%%%%%%%%%%%%%%%%%%%%%%%%%%%%%%%%%%%%%%%%%%%%%
\subsection{Convex optimization}
%%%%%%%%%%%%%%%%%%%%%%%%%%%%%%%%%%%%%%%%%%%%%%%%%%%%%%%%%%%%%%%%%%%%%%%%%%%%%%%%%%%%%%%%%%
Consider the nonlinear constrained problem
\begin{equation}\label{eq_ap_problem1}
\begin{split}
\min\limits_{\textbf{x} \in \Omega}&\textrm{ } \phi(\textbf{x})\\
\textrm{s.t. }& g_i(\textbf{x}) \le 0, i = 1,\dots,p
\end{split}
\end{equation}
under convexity and interior point assumptions as follows.
\begin{Assumption}[Assumption 5.3.1 \cite{DimitriPBertsekas1999}]\label{as_ConvexAndInterior}
The set $\Omega$ is a convex subset of $\mathbb{R}^n$ and the functions $\phi:\mathbb{R}^n \rightarrow \mathbb{R}$, $g_j:\mathbb{R}^n \rightarrow \mathbb{R}$ are convex over $\mathcal{X}$.
In addition, there exists a vector $\bar{\textbf{x}} \in \\Omega$ such that $g_i(\bar{\textbf{x}}) < 0, i = 1,\dots,p$.
\end{Assumption}

The Lagrangian function corresponding to the problem \eqref{eq_ap_problem1} is
\[\mathcal{L}(\textbf{x},\boldsymbol{\mu}) = \phi(\textbf{x}) + \boldsymbol{\mu}^T\textbf{g}(\textbf{x})\]
where $\textbf{g}(\textbf{x}) = [g_1(\textbf{x}), g_2(\textbf{x}), \dots, g_p(\textbf{x})]^T$.
Define the dual function by $\varphi(\boldsymbol{\mu}) = \inf_{\textbf{x}\in\mathcal{X}}\mathcal{L}(\textbf{x},\boldsymbol{\mu})$.
The following proposition gives a condition under which the dual function $\varphi(\boldsymbol{\mu})$ is differentiable.
\begin{Proposition}[Proposition 6.1.1 \cite{DimitriPBertsekas1999}]\label{prop_gradientproject}
Let $\phi$ and $\textbf{g}$ be continuous over $\Omega$.
Assume also that $\phi$ is strictly convex, $g_i$'s are linear and for every $\boldsymbol{\mu} \in \mathbb{R}^p$, $\mathcal{L}(\textbf{x}, \boldsymbol{\mu})$ is minimized over $\textbf{x} \in \Omega$ at a unique point $\textbf{x}_{\boldsymbol{\mu}}$.
Then, $\varphi(\boldsymbol{\mu})$ is differentiable and
\[\nabla \varphi(\boldsymbol{\mu}) = \textbf{g}(\textbf{x}_{\boldsymbol{\mu}}).\]
\end{Proposition}

The associated dual problem of \eqref{eq_ap_problem1} is given as:
\begin{equation}\label{eq_ap_dualproblem}
\max\limits_{\boldsymbol{\mu} \ge \textbf{0}}\varphi(\boldsymbol{\mu}).
\end{equation}
Let $\phi^*$ and $\varphi^*$ be the optimal values of the primal problem \eqref{eq_ap_problem1} and the dual problem \eqref{eq_ap_dualproblem}, respectively.
From the definition, we have $\varphi^* \le \phi^*$ in general.
Moreover, if the primal problem \eqref{eq_ap_problem1} satisfies the conditions stated in Assumption \ref{as_ConvexAndInterior}, the existence of Lagrange multiplier and the optimal solution of the problem \eqref{eq_ap_problem1} is guaranteed by the following propositions.
\begin{Proposition}[Proposition 5.3.1 \cite{DimitriPBertsekas1999}]\label{prop_strongdual}
Let Assumption \ref{as_ConvexAndInterior} hold for the problem \eqref{eq_ap_problem1}.
Then $\phi^* = \varphi^*$ and there exists at least one Lagrange multiplier $\boldsymbol{\mu}^*$ where 
\[\boldsymbol{\mu}^* \ge \textbf{0} \textrm{ and } \phi^* = \inf\limits_{\textbf{x}\in\mathcal{X}}\mathcal{L}(\textbf{x},\boldsymbol{\mu}^*)\]
\end{Proposition}
\begin{Proposition}[Proposition 5.1.1 \cite{DimitriPBertsekas1999}]\label{prop_optimalsolution}
Let $\boldsymbol{\mu}^*$ be a Lagrange multiplier having the properties given in Proposition \ref{prob_strongdual}. 
Then $\textbf{x}^*$ is a global minimum of the primal problem if and only if $\textbf{x}^*$ is feasible and
\begin{subequations}\label{eq_optprimal_dual}
\begin{gather}
\textbf{x}^* = \arg\min\limits_{\textbf{x}\in\mathcal{X}}\mathcal{L}(\textbf{x}, \boldsymbol{\mu}^*)\\
\mu_j^*g_j(\textbf{x}^*) = 0, j = 1,\dots,p
\end{gather}
\end{subequations}
\end{Proposition}
%%%%%%%%%%%%%%%%%%%%%%%%%%%%%%%%%%%%%%%%%%%%%%%%%%%%%%%%%%%%%%%%%%%%%%%%%%%%%%%%%%%%%%%%%%
\subsection{Alternating Direction Method of Multipliers (ADMM)}\label{subADMM}
%%%%%%%%%%%%%%%%%%%%%%%%%%%%%%%%%%%%%%%%%%%%%%%%%%%%%%%%%%%%%%%%%%%%%%%%%%%%%%%%%%%%%%%%%%
Consider the following convex optimization problem
\begin{equation}\label{eq_ADMMproblem}
\min\limits_{\textbf{u} \in \Omega_{u}, \textbf{v} \in \Omega_{v}} \textrm{ }\phi_u(\textbf{u}) + \phi_v(\textbf{v}) \textrm{ s.t. } \textbf{H}_u\textbf{u} + \textbf{H}_v\textbf{v} = \textbf{h}.
\end{equation}
where $\phi_u(\textbf{u}), \phi_v(\textbf{v})$ are convex functions, $\Omega_{u}, \Omega_{v}$ are convex sets, $\textbf{H}_u, \textbf{H}_v$ and $\textbf{h}$ are known matrices and vector.
The Lagrangian function of the problem \eqref{eq_ADMMproblem} is given by $\mathcal{L} = \phi_u(\textbf{u}) + \phi_v(\textbf{v}) + \boldsymbol{\lambda}^T\left(\textbf{H}_u\textbf{u} + \textbf{H}_v\textbf{v} - \textbf{h}\right)$ where $\boldsymbol{\lambda}$ is the dual variable associated with the equality constraint.
The alternating direction method of multipliers (ADMM) algorithm \cite{StephenBoyd2011} for solving \eqref{eq_ADMMproblem} is given by the iteration update \eqref{eq_ADMM} as follows.
\begin{subequations}\label{eq_ADMM}
\begin{align}
\textbf{v}(s+1) = \arg\min\limits_{\textbf{v} \in \Omega_{v}} \Bigl\{& \phi_v(\textbf{v}) + \frac{\rho}{2}\left|\left|\textbf{H}_u\textbf{u}(s) + \textbf{H}_v\textbf{v} - \textbf{h} - \frac{1}{\rho}\boldsymbol{\lambda}(s)\right|\right|^2 \Bigr\},\\
\textbf{u}(s+1) = \arg\min\limits_{\textbf{u} \in \Omega_{u}} \Bigl\{& \phi_u(\textbf{u}) + \frac{\rho}{2}\left|\left|\textbf{H}_u\textbf{u} + \textbf{H}_v\textbf{v}(s+1) - \textbf{h} - \frac{1}{\rho}\boldsymbol{\lambda}(s)\right|\right|^2 \Bigr\},\\
\boldsymbol{\lambda}(s+1) = \boldsymbol{\lambda}(s) - \rho \Bigl(& \textbf{H}_u\textbf{u}(s+1) + \textbf{H}_v\textbf{v}(s+1) - \textbf{h} \Bigr),
\end{align}
\end{subequations}
where $\rho > 0$ is a penalty parameter.
The following proposition summarizes some convergence results of the ADMM algorithm \cite{StephenBoyd2011}.
\begin{Proposition}\label{prop_ADMM}
As $s \rightarrow \infty$, we have
\begin{enumerate}
\item $\left|\left|\textbf{u}(s) - \textbf{u}^{opt}\right|\right| \rightarrow 0$, where $(\textbf{u}^{opt}, \textbf{v}^{opt}, \boldsymbol{\lambda}^{opt})$ corresponds to one optimal point of the convex optimization problem \eqref{eq_ADMMproblem}.
\item $\left|\left|\textbf{H}_u\textbf{u}(s) + \textbf{H}_v\textbf{v}(s) = \textbf{h}\right|\right| \rightarrow 0$
\item $\phi_u(\textbf{u}(s)) + \phi_v(\textbf{v}(s)) \rightarrow \phi^{opt}$, where $\phi^{opt}$ is the optimal cost value of the convex optimization problem \eqref{eq_ADMMproblem}.
\end{enumerate}
\end{Proposition}
%%%%%%%%%%%%%%%%%%%%%%%%%%%%%%%%%%%%%%%%%%%%%%%%%%%%%%%%%%%%%%%%%%%%%%%%%%%%%%%%%%%%%%%%%%
\bibliographystyle{IEEEtran}
\bibliography{mylib}
%%%%%%%%%%%%%%%%%%%%%%%%%%%%%%%%%%%%%%%%%%%%%%%%%%%%%%%%%%%%%%%%%%%%%%%%%%%%%%%%%%%%%%%%%%

\end{document}